\documentclass[reqno,10pt,centertags]{amsart} 
\usepackage{amsmath,amsthm,amscd,amssymb,latexsym,esint,upref,stmaryrd,
enumerate,verbatim,yfonts}
\usepackage{color}
\usepackage{hyperref} 
\newcommand*{\mailto}[1]{\href{mailto:#1}{\nolinkurl{#1}}}
\newcommand{\arxiv}[1]{\href{http://arxiv.org/abs/#1}{arXiv:#1}}

\newcommand{\bbC}{{\mathbb{C}}}

\newcommand{\bbN}{{\mathbb{N}}}

\newcommand{\bbR}{{\mathbb{R}}}
\newcommand{\bbT}{{\mathbb{T}}}

\newcommand{\bbZ}{{\mathbb{Z}}}

\newcommand{\cA}{{\mathcal A}}
\newcommand{\cB}{{\mathcal B}}
\newcommand{\cC}{{\mathcal C}}
\newcommand{\cD}{{\mathcal D}}

\newcommand{\cH}{{\mathcal H}}

\newcommand{\cJ}{{\mathcal J}}
\newcommand{\cK}{{\mathcal K}}

\newcommand{\cT}{{\mathcal T}}

\newcommand{\cX}{{\mathcal X}}

\newcommand{\gA}{{\mathfrak{A}}}

\newcommand{\gF}{{\mathfrak{F}}}
\newcommand{\gM}{{\mathfrak{M}}}

\newcommand{\beq}{\begin{equation}}
\newcommand{\enq}{\end{equation}}

\newcommand{\m}{\mu}

\DeclareMathOperator{\supp}{supp}
\DeclareMathOperator{\esssup}{ess.sup}

\DeclareMathOperator{\tr}{tr}

\DeclareMathOperator*{\slim}{s-lim}

\newcommand{\no}{\notag}
\newcommand{\lb}{\label}

\newcommand{\wti}{\widetilde}
\newcommand{\Oh}{O}

\newcommand{\hatt}{\widehat} 
\newcommand{\bi}{\bibitem}

\let\geq\geqslant
\let\leq\leqslant

\makeatletter
\def\theequation{\@arabic\c@equation}

\allowdisplaybreaks 
\numberwithin{equation}{section}

\newtheorem{theorem}{Theorem}[section]
\newtheorem{proposition}[theorem]{Proposition}
\newtheorem{lemma}[theorem]{Lemma}
\newtheorem{corollary}[theorem]{Corollary}
\newtheorem{definition}[theorem]{Definition}
\newtheorem{hypothesis}[theorem]{Hypothesis}
\newtheorem{example}[theorem]{Example}

\theoremstyle{remark}
\newtheorem{remark}[theorem]{Remark}

\begin{document}


\title[Double Operator Integral Methods and Spectral Shift Functions]{Double Operator Integral Methods Applied to Continuity of Spectral Shift Functions}

\author[A.\ Carey]{Alan Carey}  
\address{Mathematical Sciences Institute, Australian National University, 
Kingsley St., Canberra, ACT 0200, Australia 
and School of Mathematics and Applied Statistics, University of Wollongong, NSW, Australia,  2522}  
\email{\mailto{acarey@maths.anu.edu.au}}
\urladdr{\url{http://maths.anu.edu.au/~acarey/}}
  
\author[F.\ Gesztesy]{Fritz Gesztesy}  
\address{Department of Mathematics,
University of Missouri, Columbia, MO 65211, USA}
\email{\mailto{gesztesyf@missouri.edu}}
\urladdr{\url{https://www.math.missouri.edu/people/gesztesy}}

\author[G.\ Levitina]{Galina Levitina} 
\address{School of Mathematics and Statistics, UNSW, Kensington, NSW 2052,
Australia} 
\email{\mailto{g.levitina@student.unsw.edu.au}}

\author[R.\ Nichols]{Roger Nichols}
\address{Mathematics Department, The University of Tennessee at Chattanooga, 
415 EMCS Building, Dept. 6956, 615 McCallie Ave, Chattanooga, TN 37403, USA}
\email{\mailto{Roger-Nichols@utc.edu}}
\urladdr{\url{http://www.utc.edu/faculty/roger-nichols/index.php}}

\author[D.\ Potapov]{Denis Potapov}
\address{School of Mathematics and Statistics, UNSW, Kensington, NSW 2052,
Australia} 
\email{\mailto{d.potapov@unsw.edu.au}}

\author[F.\ Sukochev]{Fedor Sukochev}
\address{School of Mathematics and Statistics, UNSW, Kensington, NSW 2052,
Australia} 
\email{\mailto{f.sukochev@unsw.edu.au}}

\date{\today}
\thanks{A.C., G.L., and F.S. gratefully acknowledge financial support from the Australian
Research Council. R.N. gratefully acknowledges support from a UTC College of Arts and 
Sciences RCA Grant.} 

\dedicatory{Dedicated to the memory of Yuri G.\ Safarov (1958--2015)}


\begin{abstract} 
We derive two principal results in this note. To describe the first, assume that 
$A$, $B$, $A_n$, $B_n$, $n \in \mathbb{N}$, are self-adjoint operators in a 
complex, separable Hilbert space $\mathcal{H}$, and suppose that 
$\slim_{n \to \infty} (A_n - z_0 I_{\mathcal{H}})^{-1} = (A - z_0 I_{\mathcal{H}})^{-1}$,  
$\slim_{n \to \infty} (B_n - z_0 I_{\mathcal{H}})^{-1} = (B - z_0 I_{\mathcal{H}})^{-1}$ for 
some $z_0 \in \mathbb{C} \backslash \mathbb{R}$. Fix $m \in \mathbb{N}$, $m$ odd, 
$p \in [1,\infty)$, and assume that for all $a \in \mathbb{R} \backslash \{0\}$, 
\begin{align*} 
& T(a):= \big[( A -aiI_{\mathcal{H}})^{-m} 
- ( B -ai I_{\mathcal{H}})^{-m}\big] \in \mathcal{B}_p(\mathcal{H}),   \\ 
& T_n(a) := \big[( A_n -aiI_{\mathcal{H}})^{-m} - ( B_n -aiI_{\mathcal{H}})^{-m}\big] 
\in \mathcal{B}_p(\mathcal{H}),    \\
& \lim_{n \rightarrow \infty} \|T_n(a) - T(a)\|_{\mathcal{B}_p(\mathcal{H})} =0. 
\end{align*} 
Then for any function $f$ in the class $\mathfrak F_{m}(\mathbb{R}) \supset 
C_0^{\infty}(\mathbb{R})$ (cf.\ \eqref{1.1} for details), 
$$
\lim_{n \rightarrow \infty} \big\| [f(A_n) - f(B_n)] - [f(A)
- f(B)]\big\|_{\mathcal{B}_p(\mathcal{H})}=0.  
$$
Moreover, for each $f\in \mathfrak{F}_m(\bbR)$, $p\in [1,\infty)$, we prove the existence of constants 
$a_1,a_2 \in \mathbb{R} \backslash \{0\}$ and $C=C(f,m, a_1, a_2) \in (0,\infty)$ such that
\begin{align*}
&\big\|f(A)-f(B)\|_{\cB_p(\cH)}\leq C\big(\big\|(A-a_1iI_{\cH})^{-m} - (B-a_1iI_{\cH})^{-m}\big\|_{\cB_p(\cH)}\no\\
&\hspace*{3.4cm}+\big\|(A-a_2iI_{\cH})^{-m} - (B-a_2iI_{\cH})^{-m}\big\|_{\cB_p(\cH)}\big), \no
\end{align*}
which permits the use of differences of higher powers $m \in \mathbb{N}$ of resolvents to control 
the $\| \cdot \|_{\mathcal{B}_p(\mathcal{H})}$-norm of the left-hand side $[f(A)-f(B)]$ for 
$f \in \mathfrak{F}_{m}(\mathbb{R})$.

Our second result is concerned with the continuity of spectral shift functions $\xi(\, \cdot \,; B,B_0)$ 
associated with a pair of of self-adjoint operators $(B,B_0)$ in $\mathcal{H}$ with respect to the operator parameter $B$. For brevity, we only describe one of the consequences of our continuity results: Assume that $A_0$ and $B_0$ are fixed self-adjoint operators in 
$\mathcal{H}$, and there exists $m \in \mathbb{N}$, $m$ odd, such that,
$\big[(B_0 - z I_{\mathcal{H}})^{-m} - (A_0 - z I_{\mathcal{H}})^{-m}\big] \in 
\mathcal{B}_1\big(\mathcal{H}\big)$, 
$z \in \mathbb{C} \backslash \mathbb{R}$. For $T$ self-adjoint in $\mathcal{H}$ we denote 
by $\Gamma_m(T)$ the set of all self-adjoint operators $S$ in $\mathcal{H}$ for which the 
containment $\big[(S - z I_{\mathcal{H}})^{-m} - (T - z I_{\mathcal{H}})^{-m}\big] \in 
\mathcal{B}_1(\mathcal{H})$, $z \in \mathbb{C}\backslash \mathbb{R}$, 
holds. Suppose that $B_1\in \Gamma_m(B_0)$ and let $\{B_{\tau}\}_{\tau\in [0,1]}\subset \Gamma_m(B_0)$ 
denote a continuous path (in a suitable topology on $\Gamma_m(B_0)$, cf.\ \eqref{1.9}) from 
$B_0$ to $B_1$ in $\Gamma_m(B_0)$. If $f \in L^{\infty}(\mathbb{R})$, then
$$
\lim_{\tau\to 0^+} \|\xi(\, \cdot \, ; B_{\tau}, A_0) f 
- \xi(\, \cdot \, ; B_0, A_0) f\|_{L^1(\mathbb{R}; (|\nu|^{m+1} + 1)^{-1}d\nu)} = 0.  
$$ 
The fact that higher powers $m\in\mathbb{N}$, $m \geq 2$, of resolvents are involved, permits applications of this circle of ideas to elliptic partial differential operators in $\mathbb{R}^n$, 
$n \in \mathbb{N}$. The methods employed in this note rest on double operator integral (DOI) techniques. 
\end{abstract}

\maketitle



\section{Introduction} \lb{s1} 

We dedicate this note to the memory of Yuri Safarov (1958--2015), a gentle giant in the area 
of spectral theory, whose contribution to the field (see, for instance, the highly influential monograph 
\cite{SV97}) left an indelible impression on our community. His presence is sorely missed. 

We derive two principal results in this note. To describe the first, we introduce the class of 
functions $\gF_m(\bbR)$, $m \in \bbN$ \cite{Ya05}, by
\begin{align}
& \gF_m(\bbR) := \big\{f \in C^2(\bbR) \, \big| \, 
f^{(\ell)} \in L^{\infty}(\bbR); \text{ there exists } 
\varepsilon >0 \text{ and } f_0 = f_0(f) \in \bbC    \no \\
& \quad  \text{ such that } 
\big(d^{\ell}/d \lambda^{\ell}\big)\big[f(\lambda) - f_0 \lambda^{-m}\big] \underset{|\lambda|\to \infty}{=} 
\Oh\big(|\lambda|^{- \ell - m - \varepsilon}\big), \, \ell = 0,1,2 \big\}.     \lb{1.1}
\end{align} 
(It is implied that $f_0 = f_0(f)$ is the same as $\lambda \to \pm \infty$.) One observes that 
$\mathfrak F_m(\mathbb{R}) \supset C_0^{\infty}(\mathbb{R})$, $m \in \bbN$.

Assuming that $A$, $B$, $A_n$, $B_n$, $n \in \bbN$, are self-adjoint operators in a 
complex, separable Hilbert space $\cH$, suppose in addition that
\begin{equation}
\slim_{n \to \infty} (A_n - z_0 I_{\cH})^{-1} = (A - z_0 I_{\cH})^{-1}, \quad 
\slim_{n \to \infty} (B_n - z_0 I_{\cH})^{-1} = (B - z_0 I_{\cH})^{-1},
\end{equation}
for some $z_0 \in \bbC \backslash \bbR$. Fix $m \in \bbN$, $m$ odd, $p \in [1,\infty)$, and 
assume that for each $a \in \bbR \backslash \{0\}$, 
\begin{align} 
\begin{split} 
T(a) &:= \big[( A - aiI_{\cH})^{-m} - ( B - aiI_{\cH})^{-m}\big] \in \cB_p(\cH),   \\  
T_n(a)&:= \big[( A_n - aiI_{\cH})^{-m} - ( B_n - aiI_{\cH})^{-m}\big] \in\cB_p(\cH),   
\end{split} 
\end{align}
and 
\begin{equation}
\lim_{n \rightarrow \infty} \|T_n(a) - T(a)\|_{\cB_p(\cH)} =0.
\end{equation}
Then for any $ f \in \mathfrak F_m(\bbR)$,
\begin{equation}
\lim_{n \rightarrow \infty} \big\| [f(A_n) - f(B_n)] - [f(A)
- f(B)]\big\|_{\cB_p(\cH)}=0.   \lb{1.5}
\end{equation} 

Moreover, for each $f\in \mathfrak{F}_m(\bbR)$, $p\in [1,\infty)$, we prove the existence of constants 
$a_1,a_2 \in \mathbb{R} \backslash \{0\}$ and $C=C(f,m, a_1, a_2) \in (0,\infty)$ such that
\begin{align} 
\begin{split} 
&\big\|f(A)-f(B)\|_{\cB_p(\cH)}\leq C\big(\big\|(A-a_1iI_{\cH})^{-m} - (B-a_1iI_{\cH})^{-m}\big\|_{\cB_p(\cH)}     \\
&\hspace*{3.9cm} + \big\|(A-a_2iI_{\cH})^{-m} - (B-a_2iI_{\cH})^{-m}\big\|_{\cB_p(\cH)}\big). \lb{1.6}
\end{split} 
\end{align}
 
The estimate \eqref{1.6} is of particular interest as it permits to control the 
$\| \cdot \|_{\cB_p(\cH)}$-norm of $[f(A)-f(B)]$, $f \in \gF_m(\bbR)$, in terms of differences 
of higher powers $m \in \bbN$ of resolvents of $A$ and $B$. This is significant in applications to 
elliptic partial differential operators for which differences of sufficiently high integer powers of 
resolvents, but not necessarily the difference of resolvents itself, typically lie in the trace class 
(cf.~also our brief comments following (1.12)).

This circle of ideas is treated in detail in Sections \ref{s2} and \ref{s3}, employing the method of double operator 
integrals (DOI) (cf.\ \cite{BS73}, \cite{BS03}, \cite{Ya05}).

The second main result of this note concerns the continuity of spectral shift functions 
$\xi(\, \cdot \,; B,B_0)$ associated with a pair of of self-adjoint operators $(B,B_0)$ in $\cH$ 
(cf.\ \cite{BY93}, \cite[Ch.~8]{Ya92} for details on $\xi$) 
with respect to the operator parameter $B$. To keep the following sufficiently short, we only describe 
one of the consequences of our continuity results. We note, however, that it was precisely this consequence that was employed in recent applications to Witten index computations for certain 
classes of non-Fredholm Dirac-type operators without a mass gap in 
\cite{CGLPSZ15}--\cite{CGLS15} (see also 
\cite{CGPST15}, \cite{GLMST11}). To set 
this up, assume that $A_0$ and $B_0$ are fixed self-adjoint operators in the Hilbert space $\cH$, 
and there exists $m \in \bbN$, $m$ odd, such that,
\begin{equation}
\big[(B_0 - z I_{\cH})^{-m} - (A_0 - z I_{\cH})^{-m}\big] \in \cB_1\big(\cH\big), 
\quad z \in \bbC \backslash \bbR.       \lb{1.7} \\
\end{equation} 
Next, for $T$ self-adjoint in $\cH$, we introduce $\Gamma_m(T)$ as the set of all 
self-adjoint operators $S$ in $\cH$ for which the containment
\begin{equation}\lb{1.8}
\big[(S - z I_{\cH})^{-m} - (T - z I_{\cH})^{-m}\big] \in \cB_1(\cH), \quad z \in \bbC\backslash \bbR,
\end{equation}
holds. The family of pseudometrics
\begin{equation}\lb{1.9}
d_{m,z}(S_1,S_2) = \big\|(S_2-zI_{\cH})^{-m} - (S_1-zI_{\cH})^{-m}\big\|_{\cB_1(\cH)}, 
\quad S_1,S_2\in \Gamma_m(T), \; z\in\bbC\backslash \bbR,
\end{equation}
generates a topology, $\cT_m(\cD,T)$, on $\Gamma_m(T)$. Finally, suppose that 
$B_1\in \Gamma_m(B_0)$ and let 
\begin{equation} 
\{B_{\tau}\}_{\tau\in [0,1]}\subset \Gamma_m(B_0)
\end{equation}
denote a path from $B_0$ to $B_1$ in $\Gamma_m(B_0)$ such that $B_{\tau}$ depends 
continuously on $\tau\in [0,1]$ with respect to the topology $\cT_m(\cD,T)$. If $f \in L^{\infty}(\bbR)$, then
\begin{equation}
\lim_{\tau\to 0^+} \|\xi(\, \cdot \, ; B_{\tau}, A_0) f 
- \xi(\, \cdot \, ; B_0, A_0) f\|_{L^1(\bbR; (|\nu|^{m+1} + 1)^{-1}d\nu)} = 0,     \lb{1.11} 
\end{equation}
in particular,
\begin{equation}
\lim_{\tau \to 0^+} \int_{\bbR} \xi(\nu; B_{\tau}, A_0) d \nu \, g(\nu) 
= \int_{\bbR} \xi(\nu; B_0, A_0) d \nu \, g(\nu)      \lb{1.12}
\end{equation}
for all $g \in L^{\infty}(\bbR)$ such that 
$\esssup_{\nu \in \bbR} \big|(|\nu|^{m+1} + 1) g(\nu)\big| < \infty$.

We emphasize that in the special case $m=1$, the continuity result \eqref{1.11} for spectral shift functions with respect to trace norm convergence of resolvent differences was derived by Yafaev \cite[Lemma~8.7.5]{Ya92}. To be able to apply this to elliptic partial differential operators 
 (particularly, to Dirac and Schr\"odinger-type operators in $\bbR^n$, $n \geq 2$, cf., e.g., 
 \cite{Ya05}, \cite{Ya07}, \cite[Chs.~3, 9]{Ya10}), one typically needs $m$ sufficiently large, 
 depending on $n$ (especially, for $n \geq 4$). It was precisely this fact and concrete applications 
 to one-dimensional as 
 well as multi-dimensional Dirac-type operators without a mass gap (rendering these Dirac-type operators non-Fredholm) which are approximated by certain pseudo-differential operators, that motivated us to write this note. The Witten index for these types of non-Fredholm Dirac-type 
 operators (a concept extending the Fredholm index) is computed in terms of spectral shift 
 functions and the latter are approximated by the spectral shift functions corresponding to the 
pseudo-differential approximants. Due to limitations of space we will not go into further details at this 
point but refer to \cite{CGLPSZ15}--\cite{CGLS15} (see also \cite{GLMST11}). 
 
Results of the type \eqref{1.11} and extensions thereof are treated in detail in Section \ref{s4}, employing \cite[Lemma~8.7.5]{Ya92}, \cite{Ya05}, and particularly the results derived in 
Sections \ref{s2} and \ref{s3}.

Finally, we briefly describe some of the notation used in this note.   
The symbol $\cB_p(\cH)$, $p \in [1, \infty)$, denotes the standard $\ell^p$-based 
Schatten--von Neumann trace ideals over the complex, separable Hilbert space $\cH$, $\cB_\infty(\cH)$ denotes the ideal of compact operator in $\cH$, and if 
$\cX$ is a Banach space, $\cB(\cX)$ denotes the Banach space of all bounded, linear operators 
on $\cX$.  

The family of strongly right-continuous spectral projections associated to a self-adjoint 
operator $A$ in $\cH$ is denoted by $\{E_A(\lambda)\}_{\lambda \in \bbR}$, with 
$E_A(\lambda) = E_A((-\infty,\lambda])$, 
$\lambda \in \bbR$. 

The notation $\slim_{n\to \infty}T_n$ stands for the strong (i.e., pointwise) limit of a sequence of bounded operators $\{T_n\}_{n=1}^{\infty}$ in $\cH$.

The symbol $C_b(\bbR_+)$ represents bounded, continuous functions on $\bbR_+ = [0,\infty)$, 
and analogously for $C_b^k(\bbR_+)$, $k \in \bbN$.

\section{Norm Bounds Controlled by Powers of Resolvents via DOI} \lb{s2}

The principal aim of this section is to prove \eqref{2.52a} which shows how trace ideal bounds of resolvent powers of self-adjoint operators 
in $\cH$ control those of a sufficiently large class of functions of such operators. 

Throughout, we denote by $\cJ^{A,B}_\phi$ the linear mapping defined by the double operator integral 
\begin{equation}\lb{Z.2}
\cJ^{A,B}_\phi(T)=\int_\bbR\int_\bbR\phi(\lambda,\mu) \, d E_A(\lambda)\, T\, d E_B(\mu), 
\quad T\in \cB(\cH),
\end{equation}
where $E_A, E_B$ is spectral measures corresponding to the self-adjoint (respectively, unitary) operators $A,B$. We refer to \cite{BS03} for the precise definition and general properties of the 
double operator integrals.

It is known that if $\phi(\lambda,\mu)=a_1(\lambda)a_2(\mu)$, $(\lambda,\mu) \in \bbR^2$, for some bounded functions $a_1$ and $a_2$ on $\bbR$, then 
\begin{equation}\lb{Z.3}
\cJ_\phi^{A,B}(T)=a_1(A)Ta_2(B).
\end{equation}

Depending on the function $\phi$, the operator $\cJ^{A,B}_\phi(T)$ is bounded.  
Below we will recall a result describing the class of functions $\phi$ such that
\begin{equation}\lb{Z.4a}
\cJ^{A,B}_\phi:\cB_p(\cH)\to \cB_p(\cH),\; p \in [1, \infty), 
\quad \cJ^{A,B}_\phi:\cB(\cH)\to \cB(\cH),
\end{equation} 
is a bounded operator. We introduce 
\begin{align}
\begin{split} 
\gM_p &:= \big\{\phi \in L^{\infty}(\bbR^2;d\rho) \, \big| \, \cJ^{A,B}_\phi \in \cB(\cB_p(\cH))\big\}, 
\quad p \in [1,\infty), \\
\gM_{\infty} &:= \big\{\phi \in L^{\infty}(\bbR^2;d\rho) \, \big| \, \cJ^{A,B}_\phi \in \cB(\cB(\cH))\big\},    \lb{Z.4A} 
\end{split} 
\end{align}
where $\rho = \rho_A \otimes \rho_B$ denotes the product measure of $\rho_A$ and $\rho_B$, 
the latter are suitable (scalar-valued) control measures for $E_A$ and $E_B$, respectively. (E.g., 
$\rho_A(\cdot) = \sum_{j \in J} (e_j, E_A(\cdot) e_j)_{\cH}$, with $\{e_j\}_{j \in J}$ a complete orthonormal system in $\cH$, $J \subseteq \bbN$ an appropriate index set, and analogously 
for $\rho_B$.) In addition, we set 
\begin{equation} 
\|\phi\|_{\gM_p} := \|\cJ^{A,B}_\phi\|_{\cB(\cB_p(\cH))}, \; p \in [1,\infty), \quad 
\|\phi\|_{\gM_{\infty}} := \|\cJ^{A,B}_\phi\|_{\cB(\cB(\cH))}.  
\end{equation} 
For simplicity, we denote 
\begin{equation}
\gM:=\gM_1=\gM_\infty, 
\end{equation}
and (cf.\ \cite[Sect.~4]{BS03}) 
\begin{equation} 
\quad \|\phi\|_{\gM} := \|\phi\|_{\gM_1} = \|\phi\|_{\gM_{\infty}}, \quad \phi \in \gM.
\end{equation}

\begin{remark}\lb{rZ.1}
By interpolation, the inclusion $\phi\in\gM$ implies that $\phi\in\gM_p$ for any $p \in [1,\infty)$, 
and $\|\phi\|_{\gM_p}\leq \|\phi\|_{\gM}$, $p \in [1,\infty)$.
\hfill $\diamond$
\end{remark}

We also recall the following result.
\begin{theorem}\cite[Theorem 4.1]{BS03}\lb{tZ.2} Assume that $A$ and $B$ are self-adjoint operators in $\cH$.~If the function $\phi(\, \cdot \, , \, \cdot \,)$ admits a representation of the form
\begin{equation}\lb{Z.4}
\phi(\lambda,\mu)=\int_\Omega \alpha(\lambda,t)\beta(\mu,t) \, d\eta(t), \quad (\lambda,\mu) \in \bbR^2, 
\end{equation}
where $(\Omega,d\eta(t))$ is an auxiliary measure space and 
\begin{equation}\lb{Z.5}
C_{\alpha}^2:=\sup_{\lambda \in \bbR}\int_\Omega|\alpha(\lambda,t)|^2 \, d\eta(t)<\infty,  \quad 
C_{\beta}^2:=\sup_{\mu\in\bbR}\int_\Omega|\beta(\mu,t)|^2 \, d\eta(t)<\infty, 
\end{equation}
then $\phi\in \gM$ and 
\begin{equation} 
\|\phi\|_{\gM} \leq C_{\alpha} C_{\beta}.
\end{equation}  
\end{theorem}

In the proof of the main theorem of this section, we need two results from \cite{Ya05} and \cite{BS03}. Since these results were stated without proof in those papers, we now supply a proof for convenience of the reader.

\begin{theorem}\cite[Theorem 5.2]{BS03}\lb{tZ.3} Assume that $A$ and $B$ are self-adjoint operators in $\cH$.~If there exist $0\leq m_1<1$ and $1<m_2$ such that
\begin{equation}\lb{Z.6}
\sup_{\m\in\bbR}\int_{\bbR}\big(|\xi|^{m_1}+|\xi|^{m_2}\big)\big|\hatt{\phi}(\xi,\mu)\big|^2\, d\xi 
= C_0^2<\infty,
\end{equation}
where $\hatt{\phi}(\xi,\mu)$ stands for the partial Fourier transform of $\phi$ with respect to the first variable, 
\begin{equation}
\hatt{\phi}(\xi,\mu) = (2 \pi)^{-1} \int_{\bbR} \phi(\lambda, \mu) e^{-i \xi \lambda} \, d \lambda, 
\quad (\xi, \mu) \in \bbR^2,
\end{equation} 
then $\phi\in \gM$ and 
\begin{equation}
\|\phi\|_{\gM}\leq CC_0, 
\end{equation}
where the constant $C=C(m_1,m_2)>0$ does not depend on $E_A$ or $E_B$. 
\end{theorem}
\begin{proof}
In view of
\begin{equation}\lb{Z.8}
m_1 < 1 < m_2,
\end{equation}
one obtains 
\begin{align}\lb{Z.9}
& \int_{\bbR} \bigl( \left| \xi \right|^{m_1} + \left| \xi
\right|^{m_2} \bigr)^{-1}\,  d \xi = 2 \,
\int_0^{+\infty} \frac { d r}{\left| r
\right|^{m_1} + \left| r \right|^{m_2}} =: C \in (0, \infty).
\end{align}
That is, $f_{m_1,m_2}(\xi)=(|\xi|^{m_1}+|\xi|^{m_2})^{-1/2}$, $m_1 < 1 < m_2$, satisfies 
$f_{m_1,m_2} \in L^2(\bbR)$.  Therefore, by \eqref{Z.6} and H\"older's inequality, one obtains
\begin{align}
& \int_{\bbR} \bigl| \hatt \phi(\xi, \mu) \bigr|\, d \xi 
= \int_{\bbR} \left[ \bigl( \left| \xi \right|^{m_1} +
\left| \xi \right|^{m_2} \bigr)^{\frac 12}\, \bigl|\hatt
\phi(\xi, \mu) \bigr| \right] \bigl( \left| \xi
\right|^{m_1} + \left| \xi \right|^{m_2} \bigr)^{-\, \frac 12}\,
d \xi \no \\
& \quad \leq \bigg(\int_{\bbR}\left[ \bigl( \left| \xi \right|^{m_1} +
\left| \xi \right|^{m_2} \bigr)^{\frac 12}\, \bigl|\hatt
\phi(\xi, \mu) \bigr| \right]^2d\xi\bigg)^{1/2}\; \bigg( \int_{\bbR}\bigl( \left| \xi
\right|^{m_1} + \left| \xi \right|^{m_2} \bigr)^{-1}\,
d \xi\bigg)^{1/2} \no \\
& \quad \leq C_0  \bigg( \int_{\bbR} \bigl( \left| \xi
\right|^{m_1} + \left| \xi \right|^{m_2} \bigr)^{-1}\,  d
\xi\bigg)^{1/2},    \lb{Z.7} 
\end{align}
uniformly for $\mu \in \bbR$.  Hence, 
\begin{equation}\lb{Z.10}
\hatt \phi(\, \cdot \,, \mu) \in L^1(\bbR),
\end{equation}
and
\begin{equation}\lb{Z.11a}
\sup_{\mu\in \bbR} \big\|\hatt \phi(\, \cdot \,,\mu)\big\|_{L^1(\bbR)}<\infty.
\end{equation}
By the inverse Fourier transform  
\begin{align}\lb{Z.11}
\begin{split} 
\phi(\lambda, \mu)&= \int_{\bbR} \hatt \phi(\xi, \mu) \,
e^{i\xi \lambda}\,  d \xi \\
&= \int_{\bbR} e^{i\xi \lambda} \left[ \bigl( \left| \xi \right|^{m_1} + \left|
\xi \right|^{m_2} \bigr)^{1/2}\, \hatt \phi (\xi, \mu)
\right] \big(|\xi|^{m_1} + | \xi|^{m_2} \big)^{- 1/2} \, d\xi.    
\end{split} 
\end{align}

Next, introduce the functions
\begin{equation}\lb{Z.12}
\alpha(\lambda, t) = e^{i\lambda t}\bigl(
\left| t \right|^{m_1} + \left| t
\right|^{m_2} \bigr)^{-\, \frac 12}, \quad \beta(\mu, t)
= \bigl( \left| t \right|^{m_1} + \left| t \right|^{m_2}
\bigr)^{\frac 12}\, \hatt \phi (t, \mu).
\end{equation}
By \eqref{Z.6} and \eqref{Z.9}, the functions $\alpha$ and $\beta$ satisfy the condition of Theorem \ref{tZ.2} with respect to the measure space $(\Omega,d\eta(t)) =(\bbR,dt)$.  Hence, by Theorem \ref{tZ.2}, $\phi\in\gM$ and $\|\phi\|_\gM\leq C C_0$, where the constant $C=C(m_1,m_2)$ does not depend on the spectral measures $E_A$ and $E_B$.
\end{proof}

\begin{proposition}\cite[Proposition 5.2]{Ya05}\lb{pZ.4} Assume that $A$ and $B$ 
are self-adjoint operators in $\cH$.~Suppose that the function $K(\lambda,\mu)$ on $\bbR^2$ satisfies
\begin{equation}\lb{Z.13}
|K(\lambda,\mu)| \leq C_K <\infty, \quad (\lambda,\mu) \in \bbR^2,
\end{equation}
and is differentiable with respect to $\lambda$ with
\begin{equation}\lb{Z.14}
\bigg|\frac{\partial K(\lambda,\mu)}{\partial\lambda}\bigg|\leq \wti C_K \big(1+\lambda^2\big)^{-1}, 
\quad (\lambda,\mu) \in \bbR^2, 
\end{equation}
where the constant $\wti C_K$ is independent of $\mu$.  Assume, in addition, that for every fixed $\mu\in\bbR$
\begin{equation}\lb{Z.15}
\lim_{\lambda\to-\infty}K(\lambda,\mu) = \lim_{\lambda\to+\infty} K(\lambda,\mu),
\end{equation}
where the limits exist by \eqref{Z.14}. Then $\cJ_K^{A,B}\in \cB(\cB(\cH))$ and $\cJ_K^{A,B}\in \cB(\cB_p(\cH))$, $p \in [1,\infty)$.
\end{proposition}
\begin{proof}
By \eqref{Z.13} and  \eqref{Z.15} the function 
\begin{equation}\lb{Z.16}
k(\mu):=\lim_{\lambda\to\pm\infty}K(\lambda,\mu), \quad \mu \in \bbR,
\end{equation}
is well-defined and bounded on $\bbR$. 

We set 
\begin{equation}\lb{Z.17}
h (\lambda, \mu) := K(\lambda, \mu) - k(\mu), \quad (\lambda,\mu) \in \bbR^2,
\end{equation}
and claim that this function satisfies the conditions of Theorem \ref{tZ.3}.  Indeed, since 
\begin{equation}\lb{Z.18}
\frac {\partial h}{\partial \lambda} = \frac {\partial
K}{\partial \lambda},
\end{equation}
one infers from \eqref{Z.14} that 
\begin{equation}\lb{Z.19}
\frac {\partial h}{\partial \lambda} (\,\cdot\,, \mu) \in
L^2(\bbR),\quad \mu\in \bbR, \, \text{ with } \, 
\sup_{\m\in \bbR}\bigg\|\frac{\partial h}{\partial \lambda}(\, \cdot \,,\mu) \bigg\|_{L^2(\bbR)}<\infty.
\end{equation}
Furthermore, by the definition of the function $h$,
\begin{equation}
\lim_{\lambda
\rightarrow \pm \infty} h(\lambda, \mu) = 0,
\end{equation}
and therefore,
\begin{equation}
h(\lambda,\mu)=
\begin{cases}
\displaystyle - \int^{+\infty}_\lambda \dfrac {\partial h}{\partial \lambda} (t,
\mu)\, dt,& \lambda>0,\\[4mm]
\displaystyle \int_{-\infty}^\lambda \dfrac {\partial
h}{\partial \lambda} (t, \mu)\, dt,& \lambda<0.
\end{cases}
\end{equation}
Hence, by~\eqref{Z.14} for $\lambda>0$,
\begin{equation}
|h(\lambda,\mu)|\leq \int^{+\infty}_\lambda\Big|\frac {\partial h}{\partial \lambda} (t,
    \mu)\Big|\, dt \leq C \int^{+\infty}_\lambda (1+t^2)^{-1}\, dt,
    \end{equation}
for an appropriate constant $C>0$.  A similar estimate for $\lambda<0$ yields
\begin{equation}
h(\lambda, \mu) = O \big(| \lambda|^{-1}\big)\ \ \text{if}\ \
\lambda \rightarrow \pm \infty,
\end{equation}
uniformly for $ \mu \in \bbR$.  Hence, $ h (\cdot, \mu) \in L^2(\mathbb R) $ and by Parseval's identity, one obtains
\begin{equation}\lb{Z.20}
\sup_{\mu \in \bbR}\, \int_{\bbR}
|\xi |^2 \big| \hatt h(\xi, \mu) \big|^2\,
d \xi < \infty.
\end{equation}
That is, the function $h(\cdot, \cdot)$ satisfies the condition  Theorem~\ref{tZ.3} with 
\begin{equation}\lb{Z.21}
 m_1 = 0 \ \
\text{and}\ \ m_2 = 2. 
\end{equation}

Hence, Theorem~\ref{tZ.3} implies that the operator $\cJ_h^{A,B}:\cB(\cH)\to \cB(\cH)$ is bounded. 
Furthermore, since $K(\lambda,\mu)=h(\lambda,\mu)+k(\mu)$, $(\lambda,\mu) \in \bbR^2$, 
\eqref{Z.3} for the operator $\cJ_K^{A,B}$ implies 
\begin{equation}\lb{Z.22}
\cJ_K^{A,B}(T) = \cJ_h^{A,B}(T) + Tk(B).
\end{equation}
Since the function $k$ is bounded one infers that  the operator $\cJ_K^{A,B}$ is bounded 
on $\cB(\cH)$. Finally, Remark \ref{rZ.1} implies that the operator $\cJ_K^{A,B}$ is also bounded on any $\cB_p(\cH)$, $p\in [1,\infty)$.
\end{proof}

\begin{corollary}\lb{cZ.5}
The norms $\big\|\cJ_K^{A,B}\big\|_{\cB(\cB(\cH))}$, $\big\|\cJ_K^{A,B}\big\|_{\cB(\cB_p(\cH))}$, 
$p\in[1,\infty)$, do not depend on the spectral measures $E_A$ and $E_B$.
\end{corollary}
\begin{proof}
This follows from the proof of Proposition \ref{pZ.4} and Theorem \ref{tZ.3}.
\end{proof}

To prove the norm bounds required for the proof of Proposition \ref{prop_continuity}, we now introduce the following assumption.

\begin{hypothesis} \lb{hyp} 
Assume that $A$ and $B$ are fixed self-adjoint operators in the Hilbert space $\cH$, $p\in[1,\infty)$, and there exists $m \in \bbN$, $m$ odd, such that for all $z\in\bbC\backslash\bbR$, 
\begin{equation}\lb{2.34zaz}
\big[(B - z I_{\cH})^{-m} - (A - z I_{\cH})^{-m}\big] \in \cB_p\big(\cH\big) \; (\text{resp., } \cB(\cH)). 
\end{equation} 
\end{hypothesis}

The following construction is taken from \cite{Ya05}. Fix a bijection $\varphi: \bbR \to \bbR$ satisfying for some $c >0$ and $r>0$,
\begin{equation}
\varphi \in C^2(\bbR), \quad \varphi(\lambda) = \lambda^{m}, \; |\lambda| \geq r, 
\quad \varphi'(\lambda) \geq c, \; \lambda\in\bbR.     \lb{W3.2}
\end{equation} 
Let $r>0$ be such that $\phi(\lambda)=\lambda^m$ for $|\lambda|\geq r$. 
We choose a function $\theta\in C^2(\bbR)$ such that  $\theta(\lambda)=0$ for $|\lambda|\leq r/2$, $\theta(\lambda)=1$ for $|\lambda|\geq r$ and 
\begin{align}\label{def_f,f_0}
\frac1{\phi(\lambda)-i}&=\theta(\lambda)\frac1{\lambda^m-i}+(1-\theta(\lambda))\frac1{\phi(\lambda)-i}=:g_1(\lambda)+g_2(\lambda),\quad \lambda\in\bbR.
\end{align}
We note that $g_2\in C^2(\bbR)$ with compact support.

Thus,
\begin{equation}\label{decomposition}
(\phi(A)-iI_{\cH})^{-1}-(\phi(B)-iI_{\cH})^{-1}=g_1(A)-g_1(B)+g_2(A)-g_2(B).
\end{equation}
Next, we denote 
\begin{equation}\label{def_K,K_0}
\begin{split}
G_{1,a}(\lambda,\mu)&=\frac{g_1(\lambda)-g_1(\mu)}{(\lambda-ia)^{-m}-(\mu-ia)^{-m}},\\ 
G_{2,a}(\lambda,\mu)&=\frac{g_2(\lambda)-g_2(\mu)}{(\lambda-ia)^{-m}-(\mu-ia)^{-m}},\quad \lambda,\mu\in\bbR,
\end{split}
\end{equation}
where $a\in\bbR\backslash\{0\}.$
In \cite[Proposition 3.3]{Ya05} it is proved that there exists a (sufficiently small) $a_1\in\bbR\backslash\{0\}$, such that the function $G_{1,a_1}$ satisfies the assumption of Proposition \ref{pZ.4}. Therefore, Proposition \ref{pZ.4} implies that 
\begin{equation}\label{DOI_g_1}
g_1(A)-g_1(B)=\cJ^{A,B}_{G_{1,a_1}}\big((A-a_1iI_{\cH})^{-m} - (B-a_1iI_{\cH})^{-m}\big)
\end{equation}
and 
\begin{equation}\label{estimate_lambda^m_part}
\|g_1(A)-g_1(B)\|_{\cB_p(\cH)}\leq C_1 \big\|(A-a_1iI_{\cH})^{-m} - (B-a_1iI_{\cH})^{-m}\big\|_{\cB_p(\cH)},
\end{equation}
for some constant $C_1=C_1(a_1,m) \in (0,\infty)$ (and a corresponding estimate for 
the $\cB(\cH)$-norm). Moreover, in \cite[Proposition 3.2]{Ya05} it is proved that there exists a (sufficiently large) $a_2\in\bbR\backslash\{0\}$, such that the function $G_{2,a_2}$ satisfies the assumption of Proposition \ref{pZ.4}. Therefore, 
\begin{equation}\label{DOI_g_2}
g_2(A)-g_2(B)=\cJ^{A,B}_{G_{2,a_2}}\big((A-a_2iI_{\cH})^{-m} - (B-a_2iI_{\cH})^{-m}\big)
\end{equation}
and 
\begin{equation}\label{estimate_compact_part}
\|g_2(A)-g_2(B)\|_{\cB_p(\cH)}\leq C_2 \big\|(A-a_2iI_{\cH})^{-m} - (B-a_2iI_{\cH})^{-m}\big\|_{\cB_p(\cH)}
\end{equation}
for some constant $C_2=C_2(a_2,m) \in (0,\infty)$ (and a corresponding estimate 
for the $\cB(\cH)$-norm). We note that the independence of the constants $C_1$ and $C_2$ in \eqref{estimate_lambda^m_part} and \eqref{estimate_compact_part} of $p \in (1,\infty)$ follows from 
the fact that $G_{1,a_1}, G_{2,a_2}\in\gM$ (see Proposition \ref{pZ.4}) and Remark \ref{rZ.1}. 
 
 Combining this with \eqref{decomposition} one arrives at the following result. If $\phi$ satisfies \eqref{W3.2}, then there exist $a_1,a_2\in\bbR\backslash\{0\}$ and $C = C(a_1,a_2,m) \in (0,\infty)$ such that
\begin{align}
&\big\|(\phi(A)-iI_{\cH})^{-1}-(\phi(B)-iI_{\cH})^{-1}\big\|_{\cB_p(\cH)}\no\\
&\quad \leq C\big(\big\|(A-a_1iI_{\cH})^{-m} - (B-a_1iI_{\cH})^{-m}\big\|_{\cB_p(\cH)} 
\lb{the_estimate}\\
&\hspace*{1.3cm}+\big\|(A-a_2iI_{\cH})^{-m} - (B-a_2iI_{\cH})^{-m}\big\|_{\cB_p(\cH)}\big),   \no
\end{align}
and an analogous estimate for the uniform norm $\|\, \cdot \, \|_{\cB(\cH)}$.

Next, we introduce the class of functions for which we prove the main results of this and the next sections.

\begin{definition}\cite{Ya05}\lb{dZ.9}
Let $m\in\bbN$.  Define the class of functions $\gF_m(\bbR)$ by
\begin{align}
& \gF_m(\bbR) := \big\{f \in C^2(\bbR) \, \big| \, f^{(\ell)} \in L^{\infty}(\bbR); \text{ there exists } 
\varepsilon >0 \text{ and } f_0 = f_0(f) \in \bbC    \no \\
& \quad  \text{such that } 
\big(d^{\ell}/d \lambda^{\ell}\big)\big[f(\lambda) - f_0 \lambda^{-m}\big] \underset{|\lambda|\to \infty}{=} 
\Oh\big(|\lambda|^{- \ell - m - \varepsilon}\big), \, \ell =0,1,2 \big\}.     \lb{Z.39}
\end{align} 
$($It is implied that $f_0 = f_0(f)$ is the same as $\lambda \to \pm \infty$.$)$
\end{definition}

In particular, one notes that for all $m \in \bbN$, 
\begin{equation} 
C_0^{\infty}(\bbR) \subset \gF_m(\bbR),   \lb{Z.40a} 
\end{equation}  
and
\begin{equation}\lb{Z.40}
f(\lambda) \underset{|\lambda| \to \infty}{=} f_0\lambda^{-m}+O  \big( |\lambda|^{-m - \epsilon} \big),\quad f\in\gF_m(\bbR).
\end{equation}

Let $f\in\mathfrak{F}_m(\bbR)$ and let $\phi$ be as before (see \eqref{W3.2}). The assumptions on the functions $\phi$ and $f$ imply that $f_0:=f\circ \phi^{-1}\in \mathfrak{F}_1(\bbR)$ (see \cite{Ya05}). It follows from the  discussion before \cite[Theorem~8.7.1]{Ya92} that there is a continuously differentiable function $g$ on $\bbT$, with  $g'$ satisfying the H\"older condition with exponent $\varepsilon>0$, such that 
\begin{equation}\label{def_g}
f_0(\lambda)=g(\gamma(\lambda)),
\end{equation}
where $\gamma(\lambda)=\frac{\lambda+i}{\lambda-i}$, $\lambda\in\bbR,$ denotes the Cayley transform.

We denote $U=\gamma(\phi(A))$, $V=\gamma(\phi(B)).$ By \eqref{the_estimate}, there exist $a_1,a_2 \in \bbR \backslash \{0\}$ and a constant $C=C(a_1,a_2,m) \in (0,\infty)$ such that 
\begin{align}\label{estimate_unitary}
\|U-V\|_{\cB_p(\cH)}&=\big\|2i(\phi(A)-iI_{\cH})^{-1}-(\phi(B)-iI_{\cH})^{-1}\big\|_{\cB_p(\cH)}\no\\
&\leq 2C\big(\big\|(A-a_1iI_{\cH})^{-m} - (B-a_1iI_{\cH})^{-m}\big\|_{\cB_p(\cH)}\\
&\hspace*{1cm}+\big\|(A-a_2iI_{\cH})^{-m} - (B-a_2iI_{\cH})^{-m}\big\|_{\cB_p(\cH)}\big),  \no
\end{align}
and an analogous estimate for the uniform norm $\|\, \cdot \, \|_{\cB(\cH)}$.

Since $g'$ satisfies the H\"older condition with exponent $\varepsilon>0$,  the double operator integral $\cJ^{U,V}_{g^{[1]}}$, where 
\begin{equation}
g^{[1]}(u,v)=\frac{g(u)-g(v)}{u-v},\quad u,v\in\bbT,
\end{equation}
is a bounded operator on $\cB_p(\cH)$,\, $p\in[1,\infty)$, and on $\cB(\cH)$ \cite[Theorem~11]{BS67}. Thus,  
\begin{align}
f(A)-f(B)=f_0(\phi(A))-f_0(\phi(B))=g(U)-g(V)=\cJ^{U,V}_{g^{[1]}}(U-V),
\end{align}
and therefore, $[f(A)-f(B)] \in \cB_p(\cH)$ (resp., $[f(A)-f(B)] \in \cB(\cH)$) and
\begin{align}
&\big\|f(A)-f(B)\|_{\cB_p(\cH)}\leq \|\mathcal{J}_{g^{[1]}}\|_{\cB(\cB_p(\cH))}\|U-V\|_{\cB_p(\cH)}\no\\
&\quad \leq C\big(\big\|(A-a_1iI_{\cH})^{-m} - (B-a_1iI_{\cH})^{-m}\big\|_{\cB_p(\cH)}\label{2.52a}\\
&\hspace*{1.2cm}+\big\|(A-a_2iI_{\cH})^{-m} - (B-a_2iI_{\cH})^{-m}\big\|_{\cB_p(\cH)}\big),\quad f\in \mathfrak{F}_m(\bbR)\no
\end{align}
(and the corresponding estimate for the uniform norm $\| \, \cdot \|_{\cB(\cH)}$). Here the 
constant $C=C(f,a_1,a_2,m) \in (0,\infty)$ is independent of $p \in (1,\infty)$ (see Remark \ref{rZ.1}).

\begin{remark}Assume Hypothesis \ref{hyp} with $p \in (1,\infty)$. Then estimate \eqref{2.52a} holds for a wider class of functions $f$, and the constant $C$ can be sharpened. Indeed, assume that function $f$ on $\bbR$ is such that the function $g$ on $\bbT$ defined by \eqref{def_g} is a Lipschitz function on $\bbT$. Then combining  \cite[Theorem~2]{ACS16} and \cite[Corollary~5.5]{CM-SPS14} one obtains $[f(A)-f(B)] \in \cB_p(\cH)$ and 
\begin{align}
& \|f(A)-f(B)\|_{\cB_p(\cH)} \leq 32\Big(C_1\frac{p^2}{p-1}+9\Big)\|U-V\|_{\cB_p(\cH)}   \no \\
&\quad \stackrel{\eqref{estimate_unitary}}{\leq} 64C_2\bigg(C_1\frac{p^2}{p-1}+9\bigg) 
\big(\|(A-a_1iI_{\cH})^{-m} - (B-a_1iI_{\cH})^{-m}\big\|_{\cB_p(\cH)}  \\
& \hspace*{4.5cm} + \big\|(A-a_2iI_{\cH})^{-m} - (B-a_2iI_{\cH})^{-m}\big\|_{\cB_p(\cH)}\big),   \no
\end{align}
where the constants $C_1=C_1(f) \in (0,\infty)$ and $C_2=C_2(a_1,a_2,m) \in (0,\infty)$ are independent of $p \in (1, \infty)$. 
\end{remark}

\begin{remark} \lb{rZ.11} 
In the special case $m = 1$, an inequality similar to \eqref{2.52a} was derived for $f \in \cA(\bbR)$ in \cite{GN15} using the notion of {\it almost analytic extensions}. Specifically, it was shown in \cite{GN15} that for each fixed $z_0\in \bbC\backslash \bbR$ and each $f\in \cA(\bbR)$, there exists a constant $C=C(f,z_0) \in (0,\infty)$, independent of $p \in [1,\infty)$, such that
\begin{equation}
\| f(A) - f(B)\|_{\cB_p(\cH)} \leq C \big\|(A - z_0 I_{\cH})^{-1} - (B - z_0 I_{\cH})^{-1}\big\|_{\cB_p(\cH)}, 
\quad p \in [1,\infty).
\end{equation}
Here $\cA(\bbR)$ is defined as 
\begin{equation}
\cA(\bbR) = \bigcup_{\beta < 0} S^{\beta}(\bbR), 
\end{equation} 
with the class $S^{\beta}(\bbR)$, $\beta \in \bbR$, consisting of all functions 
$f \in C^{\infty}(\bbR)$ such that 
\begin{equation}
f^{(m)}(x) \underset{|x|\to\infty}{=} \Oh\big(\langle x\rangle^{\beta - m}\big), \quad 
m \in \bbN_0,    \lb{2.1}
\end{equation}
where $\langle z \rangle = \big(|z|^2 + 1\big)^{1/2}$, $z\in \bbC$; in particular, 
$C_0^{\infty}(\bbR) \subset \cA(\bbR)$. For the case $\cB_p(\cH)$ replaced by $\cB(\cH)$, we 
refer to \cite[Theorem~2.6.2]{Da95}. The double operator integral (DOI) techniques employed 
in the bulk of this section not only yield the stronger estimate \eqref{2.52a} for $f \in \gF_m(\bbR)$, 
but at the same time permit the use of higher powers $m \in \bbN$ of resolvents to control 
the left-hand side of \eqref{2.52a}. \hfill $\diamond$
\end{remark}

\begin{remark} \lb{r2.4}
In connection with containments of the type in \eqref{2.34zaz}, we recall that a Cauchy-type formula implies the following elementary fact (cf.~\cite[p.~210]{Ya92})): Let $S_j$, $j\in \{1,2\}$, be 
self-adjoint operators in some complex, separable Hilbert space $\cK$. If 
\begin{equation} 
\big[(S_2 - zI_{\cK})^{-m} - (S_1 - zI_{\cK})^{-m}\big]\in \cB_p(\cK), 
\quad z\in \bbC\backslash\bbR, \lb{2.36}
\end{equation}
for some $p\in [1,\infty) \cup \{\infty\}$ and some $m\in \bbN$, then
\begin{equation} 
\big[(S_2 - zI_{\cK})^{-n} - (S_1 - zI_{\cK})^{-n}\big]\in \cB_p(\cK), 
\quad z\in \bbC\backslash\bbR,\; n\geq m.    \lb{2.37} 
\end{equation}
In the case where $S_j$, $j=1,2$, are bounded from below, see also 
\cite[Proposition~8.9.2]{Ya92}.
Hence, if \eqref{2.36} holds for some $m\in \bbN$, we may, without loss of generality, assume that 
$m$ is odd (as we will in subsequent sections).  \hfill $\diamond$
\end{remark}

\section{Limiting Process for Double Operator Integrals}  \lb{s3} 

The main purpose of this section is to prove Theorem \ref{tZ.10}.

Let $A_n,B_n, A,B$ be self-adjoint in the Hilbert space $\cH$. 
We recall the definition of the classes $\gA_r^s(E_A)$ and $\gA_l^s(E_B)$ (cf., e.g., \cite[p.~40]{BS73}).  Suppose $\phi(\, \cdot \,, \, \cdot \,)$ admits a representation of the form
\begin{equation}\lb{Z.26}
\phi(\lambda,\mu)=\int_\Omega \alpha(\lambda,t)\beta(\mu,t) \, d\eta(t),  \quad 
(\lambda,\mu) \in \bbR^2, 
\end{equation}
where $(\Omega,d\eta(t))$ is an auxiliary measure space and 
\begin{equation}\lb{Z.27}
C_{\alpha}^2 :=\sup_{\lambda \in \bbR}\int_\Omega|\alpha(\lambda,t)|^2 \, d\eta(t)<\infty,   \quad 
C_{\beta}^2 :=\sup_{\mu\in\bbR}\int_\Omega|\beta(\mu,t)|^2 \, d\eta(t)<\infty.
\end{equation}
Set 
\begin{align}\lb{Z.28}
\begin{split}
a(t) &:= \int_\bbR\alpha(\lambda,t) \, dE_A(\lambda), \quad b(t) := \int_\bbR\beta(\mu,t) \, dE_B(\mu), 
\\ 
a_n(t) &:= \int_\bbR\alpha(\lambda,t) \, dE_{A_n}(\lambda), \quad b_n(t) := \int_\bbR\beta(\mu,t) \, dE_{B_n}(\mu),\quad n\in\bbN, 
\end{split}
\end{align}
and introduce
\begin{equation}\lb{Z.29}
\begin{split}
\varepsilon_n(v,\alpha)&=\bigg[\int_\Omega \|a_n(t) v - a(t) v\|^2 \, d\eta(t)\bigg]^{1/2}, \\
\delta_n(v,\beta)&=\bigg[\int_\Omega \|b_n(t) v - b(t) v\|^2 \, d\eta(t)\bigg]^{1/2},\quad 
n\in\bbN,\; v \in \cH,
\end{split}
\end{equation}
and 
\begin{align}
& \gA_r^s(E_A) := \big\{\phi \, \text{in} \, \eqref{Z.26} \, \big| \, \lim_{n\to\infty}\varepsilon_n(v,\alpha)=0,\, 
v \in \cH\big\},   \lb{Z.30} \\
& \gA_l^s(E_B) := \big\{\phi \, \text{in} \, \eqref{Z.26} \, \big| \, \lim_{n\to\infty}\delta_n(v,\alpha)=0,\, 
v \in \cH\big\}. 
\lb{Z.31} 
\end{align}
 If $A_n,B_n, A,B$ are unitary operators on $\cH$, the classes 
$\gA_r^s(E_A), \gA_l^s(E_A)$ are introduced similarly.

We note that the definitions of the classes $\gA_r^s(E_A), \gA_l^s(E_A)$ impose certain restrictions on convergences $A_n \longrightarrow A$ and $B_n \longrightarrow B$ as well as on the properties of the function $\phi$, given in \eqref{Z.26}.

\begin{proposition}\label{fA_linear}
If $\phi,\psi \in \gA_r^s(E_A)$ $($respectively, $\phi,\psi \in \gA_l^s(E_B)$$)$, then $(\phi+\psi) \in \gA_r^s(E_A)$ $($respectively, $(\phi+\psi) \in \gA_l^s(E_B)$$)$. 
\end{proposition}
\begin{proof}We prove the assertion only for the set $\gA_r^s(E_A)$, since for the set $\gA_l^s(E_B)$ the proof is similar. 

Let the functions $\phi$ and $\psi$ have the representations
\begin{equation}
\phi(\lambda,\mu)=\int_{\Omega_1} \alpha_1(\lambda,t)\beta_1(\mu,t)\,d\eta_1(t), \quad 
\psi(\lambda,\mu)=\int_{\Omega_2} \alpha_2(\lambda,t)\beta_2(\mu,t)\,d\eta_2(t),
\end{equation}
for some measure spaces $(\Omega_i,d\eta_j(t)), $ and functions $\alpha_j,\beta_j$, $j\in \{1,2\}$.

Let $(\Omega,\Sigma ,d\eta(t)) $ be the direct sum
of the measure spaces $(\Omega_1,d\eta_1(t)) $ and \, $(\Omega_2,d\eta_2(t))$ (so $\Omega=\Omega_{1}\sqcup \Omega_{2}$, the
disjoint union of $\Omega_{1}$ and $\Omega_{2}$, etc.). Define the function 
\begin{equation}
\alpha \left( \lambda ,t\right) =
\begin{cases}
\alpha _{1}\left( \lambda ,t\right), & t\in \Omega_{1}, \\ 
\alpha _{2}\left( \lambda ,t\right), & t\in \Omega_{2}.
\end{cases}
\end{equation}
Evidently, the function $\alpha$ satisfies condition \eqref{Z.27}.   In addition,
\begin{align}
a_n(t)=
\begin{cases}
\displaystyle \int_\bbR\alpha _{1}\left( \lambda ,t\right)dE_{A_n}(t) = a_n^{(1)}(t), & t\in \Omega_{1}, \\[3mm]
\displaystyle \int_\bbR\alpha _{2}\left( \lambda ,t\right)dE_{A_n}(t) = a_n^{(2)}(t), & t\in \Omega_{2},
\end{cases}
\end{align}
and 
\begin{equation}
a(t)=
\begin{cases}
\displaystyle \int_\bbR\alpha _{1}\left( \lambda ,t\right)dE_A(t) = a^{(1)}(t), & t\in \Omega_{1}, \\[3mm]
\displaystyle \int_\bbR\alpha _{2}\left( \lambda ,t\right)dE_A(t) = a^{(2)}(t), & t\in \Omega_{2},
\end{cases}
\end{equation}
where  $a_n^{(j)}(\cdot)$ and $a^{(j)}(\cdot)$ denote the operators defined by  \eqref{Z.28} with respect to the functions $\alpha_j$, $j\in \{1,2\}$.  Hence, for every fixed $v \in \cH$, 
\begin{align}
\varepsilon_n(v,\alpha)&=\bigg(\int_\Omega\|a_n(t) v - a(t) v\|^2d\eta(t)\bigg)^2\no\\
&\leq \bigg(\int_{\Omega_1}\big\|a_n^{(1)}(t) v - a^{(1)}(t) v\big\|^2d\eta_1(t)\bigg)^2  \no \\
& \quad + \bigg(\int_{\Omega_2}\big\|a_n^{(2)}(t) v - a^{(2)}(t) v\big\|^2d\eta_2(t)\bigg)^2\no\\
&=\varepsilon_n(v,\alpha_1)+\varepsilon_n(v,\alpha_2)\underset{n \to \infty}{\longrightarrow} 0.  
\end{align}
Thus, $(\phi+\psi)\in\gA_r^s(E_A)$.
\end{proof}

Our proof of Theorem \ref{tZ.10} is based on the following result in \cite{BS73}.

\begin{proposition}\cite[Proposition 5.6]{BS73}\lb{pZ.7}
Let $\phi\in \gA_r^s(E_A)\cap\gA_l^s(E_B)$. Then for any  
$T\in\cB_p(\cH)$, $p \in [1,\infty)$, 
\begin{align}
& \lim_{n \to \infty}\big\|\cJ_\phi^{A_n,B_n}(T)-\cJ_\phi^{A,B}(T)\big\|_{\cB_p(\cH)} = 0, 
\quad p \in [1,\infty).    \lb{Z.32} 
\end{align} 
\end{proposition}

In order to formulate the main results of this section later on, we introduce the following assumption.

\begin{hypothesis} \lb{hZ.7}
Let $A$, $B$, $A_n$, $B_n$, $n \in \bbN$, be 
self-adjoint operators in a separable Hilbert space $\mathcal H$ and suppose that
\begin{equation}
\slim_{n \to \infty} (A_n - z_0 I_{\cH})^{-1} = (A - z_0 I_{\cH})^{-1}, \quad 
\slim_{n \to \infty} (B_n - z_0 I_{\cH})^{-1} = (B - z_0 I_{\cH})^{-1},      \lb{2.36A} 
\end{equation}
for some $z_0 \in \bbC \backslash \bbR$ $($cf.\ \cite[Theorem~VIII.19\,(b)]{RS80}$)$.
\end{hypothesis}

\begin{lemma}\lb{lZ.8}
Assume Hypothesis \ref{hZ.7}. If a function $\phi(\cdot, \cdot)$ satisfies the condition of Theorem~\ref{tZ.3}, then 
$\phi \in \gA_r^s (E_A)$. 
\end{lemma}
\begin{proof}
This argument is based on the proof of Theorem~\ref{tZ.3}.
Let $(\Omega,d\eta(t))=(\bbR,  dt) $  and let $ \alpha(\lambda,
t) = e^{i \lambda t}\big(|t|^{m_1}+|t|^{m_2}\big)^{-1/2}$. If $v \in \cH$, then 
\begin{equation}\lb{Z.33}
 \varepsilon_n(v, \alpha) =\bigg[ \int_{\bbR} \Big(|t|^{m_1}+|t|^{m_2}\Big)^{-1} \big\| e^{itA_n}  
v - e^{itA} v  
\big\|_{\cH}^2 \, dt \bigg]^{\frac 12},\quad n\in \bbN.
\end{equation}

Fix $\delta > 0$.  Since $\int_\bbR\big(|t|^{m_1}+|t|^{m_2}\big)^{-1} \, dt < \infty$ 
(cf., eq.~\eqref{Z.9}), there exists $ R > 0 $ such that 
\begin{equation}\lb{Z.34}
\int_{|t|>R}\big(|t|^{m_1}+|t|^{m_2}\big)^{-1} \, dt<\delta.
\end{equation}
On the other hand, since the family of functions $\big\{ e^{i\lambda t} \big\}_{t \in [-R,R]}$ 
is uniformly continuous, \cite[Theorem~VIII.21]{RS80} and the comment following its proof 
guarantees for each $v \in \cH$, 
\begin{equation}\lb{Z.36}
\lim_{n \rightarrow 0} \, \big\|e^{itA_n} v - e^{itA} v \big\|_{\cH} = 0,
\end{equation}
uniformly in $t \in [-R, R]$.  Therefore, for each $v \in \cH$, there exists $N\in \bbN$ such that
\begin{equation}\lb{Z.39a}
\big\|e^{itA_n} v - e^{itA} v \big\|_{\cH}<\delta,\quad n\geq N, \;  t\in [-R,R].
\end{equation}
Hence, for every $v \in \cH$, 
\begin{align}
\lim_{n \rightarrow \infty}\varepsilon_n(v, \alpha) & \leq \lim_{n
\rightarrow \infty} \bigg[ \int_{\left| t \right| \leq R} \big\|
e^{itA_n} v - e^{itA} v \big\|_{\cH}^2 \, d t
\bigg]^{\frac 12} \lb{Z.37}\\ 
&\quad+ \lim_{n \rightarrow \infty} \bigg[
\int_{\left| t \right| > R} \big\| e^{itA_n} v - e^{itA} v 
\big\|_{\cH}^2 \, d t \bigg]^{\frac 12} \no \\
& \leq 2 \delta \| v \|_{\cH}.\no 
\end{align}
Since $\delta > 0$ was arbitrary, one concludes 
\begin{equation}\lb{Z.38}
\lim_{n \rightarrow\infty} \varepsilon_n(v, \alpha) = 0,\quad v \in \mathcal H.
\end{equation}
\end{proof}

The next corollary is an immediate consequence of Lemma \ref{lZ.8} and Proposition \ref{pZ.4}.

\begin{corollary}\label{in_A_r}Assume Hypothesis \ref{hZ.7}.~If a function $K$ on $\bbR^2$ 
satisfies the assumption of Proposition \ref{pZ.4}, then $K\in \gA_r^s(E_B)$.
\end{corollary}
\begin{proof}
As in the proof of Proposition \ref{pZ.4} (see \eqref{Z.16} and \eqref{Z.17}), we set 
\begin{equation}
k(\mu)=\lim_{\lambda\to\pm \infty}K(\lambda,\mu),\quad h(\lambda,\mu)=K(\lambda,\mu)-k(\mu),\quad \lambda,\mu\in\bbR,
\end{equation}
and write
\begin{equation}\label{K_via_h}
K(\lambda,\mu)=h(\lambda,\mu)-k(\mu).
\end{equation}
As established in the course of the proof of Proposition \ref{pZ.4}, the function $h$ satisfies the assumption of Theorem \ref{tZ.3}.
Therefore, by Lemma \ref{lZ.8} we have $h\in \gA_r^s(E_A)$. In addition, for the function $\phi(\lambda,\mu):=k(\mu)$ we can write
\begin{equation}
\phi(\lambda,\mu)=\int_\bbR \alpha(\lambda,t)\beta(\mu,t)\, dm(t),
\end{equation}
where $\alpha(\lambda,t)=1$, $\beta(\mu,t)=k(\mu)$, and $m$ is the measure defined on the $\sigma$-algebra $2^\bbR$ by setting 
\begin{equation}
m(A)=\begin{cases} 1,& 0\in A,\\ 
0,&\text{ otherwise}.
\end{cases}
\end{equation}
Since for the function $\alpha(\lambda,t)=1$, the corresponding operators $a(t)$ and $a_n(t)$, defined in \eqref{Z.28} are just the identity operator, it is clear that the function $\phi$ belongs to the class $\gA_r^s(E_A).$ Hence, equality \eqref{K_via_h} combined with Proposition \ref{fA_linear} implies that $K\in\gA_r^s(E_A)$.  
\end{proof}

To proceed further, we now strengthen the assumptions on the operators $A_n,A$ and $B_n,B$, 
$n\in\bbN$, as follows.
\begin{hypothesis}
\label{hZ.7.1}
In addition to Hypothesis \ref{hZ.7} we assume that for some $m \in \bbN$, $m$ odd, $p \in [1,\infty)$,  and every $a\in\bbR \backslash \{0\}$, 
\begin{align} 
\begin{split} 
T(a) &:= \big[( A + iaI_{\cH})^{-m}
- ( B + iaI_{\cH})^{-m}\big] \in \cB_p(\cH),   \\  
T_n(a) &:= \big[( A_n + iaI_{\cH})^{-m} - ( B_n + iaI_{\cH})^{-m}\big] \in\cB_p(\cH),   \lb{Z.41} 
\end{split} 
\end{align}
and 
\begin{equation}
\lim_{n \rightarrow \infty} \|T_n(a) - T(a)\|_{\cB_p(\cH)} =0. 
    \lb{Z.42}
\end{equation} 
\end{hypothesis}

With this hypothesis in hand, the following theorem is the main result of this section.

\begin{theorem}\lb{tZ.10}
Assume Hypothesis \ref{hZ.7.1}.~Then for any function $ f \in \mathfrak F_m(\bbR)$,  
\begin{equation}
\lim_{n \rightarrow \infty} \big\| [f(A_n) - f(B_n)] - [f(A)
- f(B)]\big\|_{\cB_p(\cH)}=0.   \lb{Z.43}
\end{equation} 
\end{theorem}
\begin{proof}
Fix a bijection $\phi:\bbR\to\bbR$, satisfying \eqref{W3.2}. The proof is divided into two steps: \\[1mm] 
\noindent   
{\bf Step 1.}  In this step we prove that 
\begin{equation}\label{resolvent_phi_conv}
\begin{split}
\lim_{n\to\infty}&\big\| \big[(\phi(A_n)-iI_{\cH})^{-1}-(\phi(B_n)-iI_{\cH})^{-1}\big]\\
& \;\; -\big[(\phi(A)-iI_{\cH})^{-1}-(\phi(B)-iI_{\cH})^{-1}\big]\big\|_{\cB_p(\cH)}=0.
\end{split}
\end{equation}
Let $g_1$, $g_2$ be as in \eqref{def_f,f_0}. By \eqref{decomposition} one infers   
\begin{align}
&\big[(\phi(A_n)-iI_{\cH})^{-1}-(\phi(B_n)-iI_{\cH})^{-1}\big]-\big[(\phi(A)-iI_{\cH})^{-1}-(\phi(B)-iI_{\cH})^{-1}\big]\no\\
&\quad =\big[g_1(A_n)-g_1(B_n)-g_1(A)+g_1(B)\big]+\big[g_2(A_n)-g_2(B_n)-g_2(A)+g_2(B)\big].
\end{align}
Thus, to prove the assertion of step 1 it suffices to show that 
\begin{equation}
\begin{split}
\lim_{n\to\infty}\big\|g_1(A_n)-g_1(B_n)-g_1(A)+g_1(B)\big\|_{\cB_p(\cH)}=0,&\\
\lim_{n\to\infty}\big\|g_2(A_n)-g_2(B_n)-g_2(A)+g_2(B)\big\|_{\cB_p(\cH)}=0.&
\end{split}
\end{equation}
Since the proofs of these assertions are very similar, we prove the first one only. 

Let $G_{1,a}$ be the function defined by \eqref{def_K,K_0}. It is proved in \cite[Proposition 3.3]{Ya05} that there exists $0\neq a_1\in\bbR$ such that the function $G_{1,a_1}$ satisfies the assumption of Proposition \ref{pZ.4}. Thus, (see the notation \eqref{Z.41}),  
\begin{align}
&g_1(A_n)-g_1(B_n)-g_1(A)+g_1(B)\no\\
&\quad =\cJ^{A_n,B_n}_{G_{1,a_1}}(T_n(a_1))-\cJ^{A,B}_{G_{1,a_1}}(T(a_1))\no\\
&\quad =\cJ^{A_n,B_n}_{G_{1,a_1}}(T_n(a_1)-T(a_1))+\cJ^{A_n,B_n}_{G_{1,a_1}}(T(a_1))-\cJ^{A,B}_{G_{1,a_1}}(T(a_1)).\label{step1_into_two}
\end{align}
Next, we prove the convergence of each term on the right hand side of \eqref{step1_into_two} separately. 

For the first term on the right-hand side of \eqref{step1_into_two},
Proposition \ref{pZ.4} and Corollary \ref{cZ.5} imply that $ \cJ_{G_{1,a_1}}^{A_n,B_n} \in
\cB(\cB_p(\cH)) $ uniformly for $n \in \mathbb N. $ Hence,
by~\eqref{Z.42}, one obtains
\begin{equation}\lb{Z.53}
\lim_{n \rightarrow\infty} \big\|\cJ_{G_{1,a_1}}^{A_n,B_n}\big( T_n(a_1) - T(a_1) \big)\big\|_{\cB_p(\cH)} = 0.
\end{equation}

For the second term on the right-hand side of \eqref{step1_into_two}  we claim that $G_{1,a_1}\in \gA_r^s(E_A)\cap  \gA_l^s(E_B).$ Since by definition of $G_{1,a_1}$, $G_{1,a_1}(\lambda, \mu) = G_{1,a_1}(\mu,\lambda),$ it suffices to show that  $G_{1,a_1}\in \gA_r^s(E_A).$ The latter inclusion follows from the fact that the function $G_{1,a_1}$ satisfies the assumptions of Proposition \ref{pZ.4} and hence also of Corollary \ref{in_A_r}, that is, $G_{1,a_1}\in \gA_r^s(E_A)\cap  \gA_l^s(E_B)$, as required.

Thus, Proposition \ref{pZ.7} implies that
\begin{equation}
\lim_{n\to\infty}\big\|\cJ_{G_{1,a_1}}^{A_n,B_n}(T(a_1))-\cJ_{G_{1,a_1}}^{A,B}(T(a_1))\big\|_{\cB_p(\cH)}=0,
\end{equation}
concluding the proof of step 1. \\[1mm]  
\noindent 
{\bf Step 2.} Denote by $\gamma(\lambda)=\frac{\lambda+i}{\lambda-i}$, $\lambda\in\bbR,$ the Cayley transform. We set
\begin{equation}
U_n:=\gamma(\phi(A_n)), \; n \in \bbN, \quad U:=\gamma(\phi(A)), 
\end{equation}
and 
\begin{equation}
V_n:=\gamma(\phi(B_n)), \; n \in \bbN, \quad V:=\gamma(\phi(B)).
\end{equation}
Since $U_n-U=2i\big((\phi(A_n)-iI_{\cH})^{-1}-(\phi(A)-iI_{\cH})^{-1}\big)$ and by \cite[Theorem VIII.20]{RS80} $\slim_{n \to \infty}(\phi(A_n)-iI_{\cH})^{-1}=(\phi(A)-iI_{\cH})^{-1}$, one concludes that 
$\slim_{n \to \infty} U_n=U$, and similarly,  $\slim_{n \to \infty} V_n=V$. Furthermore, the convergence \eqref{resolvent_phi_conv} implies that 
\begin{equation}
\lim_{n\to\infty}\|U_n-V_n-U+V\|_{\cB_p(\cH)}=0.
\end{equation}

Let $f\in\mathfrak{F}_m(\bbR)$. The assumptions on the functions $\phi$ and $f$ imply that $f_0:=f\circ \phi^{-1}\in \mathfrak{F}_1(\bbR)$ (see \cite{Ya05}). It follows from the  discussion before \cite[Theorem 8.7.1]{Ya92} that there is a continuously differentiable function $g$ on $\bbT$, with  $g'$ satisfying the H\"older condition with exponent $\varepsilon>0$, such that 
\begin{equation}
f_0(\lambda)=g(\gamma(\lambda)).
\end{equation}

One confirms that 
\begin{equation}\label{dif_f_via_g}
f(A_n)-f(B_n)=f_0(\phi(A_n))-f_0(\phi(B_n))=g(U_n)-g(V_n),
\end{equation}
and
\begin{equation}
f(A)-f(B)=f_0(\phi(A))-f_0(\phi(B))=g(U)-g(V).
\end{equation}
Thus, to prove the convergence \eqref{Z.43} it suffices to show that 
\begin{equation}\label{conv_for_unitaries}
\lim_{n\to\infty}\big\|[g(U_n)-g(V_n)]-[g(U)-g(V)]\big\|_{\cB_p(\cH)}=0.
\end{equation}

Since $g'$ satisfies the H\"older condition with exponent $\varepsilon>0$,  the double operator integrals $\cJ^{U_n,V_n}_{g^{[1]}}$, $\cJ^{U,V}_{g^{[1]}}$, where 
\begin{equation}
g^{[1]}(u,v)=\frac{g(u)-g(v)}{u-v},\quad u,v\in\bbT,
\end{equation}
are bounded operators on $\cB_p(\cH)$,\, $p\in[1,\infty)$, with uniformly bounded norms 
(with respect to $n$) \cite[Theorem~11]{BS67}. Thus,  
\begin{align}
[g(U_n)-g(V_n)]-[g(U)-g(V)]=\cJ^{U_n,V_n}_{g^{[1]}}(U_n-V_n)-\cJ^{U,V}_{g^{[1]}}(U-V)\no\\
=\cJ^{U_n,V_n}_{g^{[1]}}(U_n-V_n-U+V)+(\cJ^{U_n,V_n}_{g^{[1]}}(U-V)-\cJ^{U,V}_{g^{[1]}}(U-V)).
\end{align}
Since $[U_n-V_n-U+V] \underset{n \to \infty}{\longrightarrow} 0$ in $\cB_p(\cH)$-norm, and the norms 
$\big\|\cJ^{U_n,V_n}_{g^{[1]}}\big\|_{\cB(\cB_p(\cH))}$ are uniformly bounded, one obtains 
\begin{equation}
\lim_{n\to\infty} \big\|\cJ^{U_n,V_n}_{g^{[1]}}(U_n-V_n-U+V)\big\|_{\cB_p(\cH)}=0.
\end{equation}

Moreover, since $g'$ satisfies the H\"older condition with exponent $\varepsilon>0$, 
a combination of \cite[Proposition 7.5]{BS73} and \cite[Theorem 5.9]{BS73}, as well as the discussion following the latter theorem, implies that $g^{[1]}$
belongs to the class $\gA_l^s(E_V)\cap\gA_r^s(E_U)$ and therefore, by Proposition \ref{pZ.7}, one infers 
\begin{equation}
\lim_{n\to\infty}\big\|\cJ^{U_n,V_n}_{g^{[1]}}(U-V)-\cJ^{U,V}_{g^{[1]}}(U-V)\big\|_{\cB_p(\cH)}=0.
\end{equation}
Thus, \eqref{conv_for_unitaries} holds, concluding the proof.
\end{proof}

\section{Continuity of $\xi(\, \cdot \,; B,B_0)$ With Respect to $B$} \lb{s4}

In this section we apply a continuity result for spectral shift functions $\xi(\, \cdot \,; B,B_0)$ 
with respect to the operator parameter $B$ in terms of trace norm convergence of resolvents 
derived by Yafaev \cite[Lemma~8.7.5]{Ya92} and extend it to the case where powers of resolvents converge, employing Sections \ref{s2} and \ref{s3} and the treatment in \cite{Ya05}.

Throughout this section, we suppose the following set of assumptions:  

\begin{hypothesis} \lb{hW3.0} 
Assume that $A_0$ and $B_0$ are fixed self-adjoint operators in the Hilbert space $\cH$, and there exists $m \in \bbN$, $m$ odd, such that,
\begin{equation}
\big[(B_0 - z I_{\cH})^{-m} - (A_0 - z I_{\cH})^{-m}\big] \in \cB_1\big(\cH\big), 
\quad z \in \bbC \backslash \bbR.       \lb{W3.1} \\
\end{equation} 
\end{hypothesis}

We denote by $\varphi: \bbR \to \bbR$ a bijection satisfying for some $c >0$,
\begin{equation}
\varphi \in C^2(\bbR), \quad \varphi(\lambda) = \lambda^{m}, \; |\lambda| \geq 1, 
\quad \varphi'(\lambda) \geq c.     \lb{W3.20}
\end{equation}
Then \cite[Theorem~2.2]{Ya05} implies that 
\begin{equation}
\big[(\varphi(B_0) - i I_{\cH})^{-1} - (\varphi(A_0) - i I_{\cH})^{-1}\big] \in \cB_1(\cH).  \lb{W3.4} 
\end{equation}
Following \cite{Ya05}, one thus introduces the class of spectral shift functions for the pair $(B_0,A_0)$ 
(cf.\ \cite{BY93}, \cite[Ch.~8]{Ya92} for details) via
\begin{equation}
\xi(\nu; B_0,A_0) = \xi(\varphi(\nu); \varphi(B_0), \varphi(A_0)), \quad \nu \in \bbR,   \lb{W3.5}
\end{equation}
implying
\begin{equation}
\xi(\, \cdot \, ; B_0,A_0) \in L^1\big(\bbR; (|\nu|^{m + 1} +1)^{-1} d\nu\big)    \lb{W3.6} 
\end{equation}
since upon introducing the new variable
\begin{equation}
\mu = \varphi(\nu) \in \bbR, \quad \nu \in \bbR,    \lb{W3.7} 
\end{equation}
the inclusion \eqref{W3.4} yields 
\begin{equation}
\xi(\, \cdot \,; \varphi(B_0), \varphi(A_0)) \in L^1\big(\bbR; (|\mu|^2 + 1)^{-1}d\mu\big).   \lb{W3.8} 
\end{equation}
Taking into account the change of variables \eqref{W3.7}, the corresponding trace formula then is of the form
\begin{align}
\tr(f(B_0)-f(A_0))&=\tr\big((f\circ \phi^{-1})(\phi(B_0))-(f\circ \phi^{-1})(\phi(A_0))\big)\no\\
&=\int_\bbR d\mu \, (f\circ \phi^{-1})'(\mu) \, \xi(\mu; \varphi(B_0), \varphi(A_0))   \no \\
&=\int_\bbR d\nu \, f'(\nu) \, \xi(\nu; B_0, A_0),\quad f \in \gF_m(\bbR),   
\end{align}
where the second equality follows from Krein's trace formula for resolvent comparable operators, that is, pairs of self-adjoint operators whose resolvent difference is trace class (see, e.g., 
\cite[Ch.~8]{Ya92}); the fact that the function $f\circ \phi^{-1}$ satisfies Krein's condition, 
that is, $f\circ \phi^{-1} \in\gF_1(\bbR)$, is guaranteed by \eqref{W3.2}.

If $S$ and $T$ are self-adjoint operators in $\cH$ and for some 
$z_0\in \bbC\backslash \bbR$,
\begin{equation}\lb{3.11ab}
[(S-z_0I_{\cH})^{-1} - (T-z_0I_{\cH})^{-1}]\in \cB_1(\cH),
\end{equation}
then actually
\begin{equation}
[(S-zI_{\cH})^{-1} - (T-zI_{\cH})^{-1}]\in \cB_1(\cH),\quad z\in \bbC\backslash \bbR,
\end{equation}
a fact which follows from the well-known resolvent identity (see, e.g., \cite[p.~178]{We80}), 
\begin{align}
& (S - z I_{\cH})^{-1} - (T - z I_{\cH})^{-1} = (S - z_0 I_{\cH})(S - z I_{\cH})^{-1}   \no \\
& \quad \times \big[(S - z_0 I_{\cH})^{-1} - (T - z_0 I_{\cH})^{-1}\big] 
(T - z_0 I_{\cH})(T - z I_{\cH})^{-1},     \\
& \hspace*{6.15cm} z, z_0 \in \rho(T_1) \cap \rho(T_2).    \no 
\end{align}
However, an analogous result cannot hold for higher powers of the resolvent as the following 
remarkably simple example illustrates.

\begin{example}
Suppose $\cH$ is an infinite-dimensional Hilbert space, and let $P_j\in \cB(\cH)$, $j\in \{1,2\}$, be infinite-dimensional orthogonal projections with 
\begin{equation}
P_1P_2 = 0\quad \text{and}\quad P_1 + P_2 = I_{\cH}.
\end{equation}
Set
\begin{equation}
A = \sqrt{3}(P_1 + P_2), \quad B = \sqrt{3}(P_1 - P_2).
\end{equation}
Evidently, $A^2=B^2=3I_{\cH}$, and
\begin{align}
(A-iI_{\cH})^3 &= A^3 - 3iA^2 +3(-i)^2A - i^3I_{\cH} = - 8 i I_{\cH}. 
\end{align}
Similarly, one obtains $(B-iI_{\cH})^3 = -8iI_{\cH}$, and consequently, 
\begin{equation}
(A-iI_{\cH})^{-3} - (B-iI_{\cH})^{-3} = 0\in \cB_1(\cH).
\end{equation}
However, if $z\in \bbC\backslash \{i\}$, then
\begin{equation}\lb{3.18aaa}
(A+zI_{\cH})^3 = A^3 + 3zA^2 + 3z^2A+z^3I_{\cH}
\end{equation}
Taking, for example, $z=3i$ in \eqref{3.18aaa}, one computes
\begin{equation}
(A+zI_{\cH})^3 = A(A^2+3z^2I_{\cH}) + z(3A^2+z^2I_{\cH}) = -24A,
\end{equation}
and similarly,
\begin{equation}
(B+3iI_{\cH})^3 = -24 B.
\end{equation}
Computing inverses, one infers
\begin{align}
(A+3iI_{\cH})^{-3} &= -\frac{1}{24}A^{-1} = -\frac{1}{24\sqrt{3}}(P_1 + P_2),\\
(B+3iI_{\cH})^{-3} &= -\frac{1}{24}B^{-1} = -\frac{1}{24\sqrt{3}}(P_1 - P_2),
\end{align}
so that
\begin{equation}
(A+3iI_{\cH})^{-3} - (B+3iI_{\cH})^{-3} = -\frac{1}{12\sqrt{3}}P_2\notin \cB_{\infty}(\cH),
\end{equation}
due to the fact that $P_2$ is an infinite-dimensional projection in $\cH$.  
\end{example}

Due to these reasons we are assuming the trace class hypothesis \eqref{W3.1} for {\it all} 
$z \in \bbC\backslash\bbR$, whenever $m \geq 2$.

\begin{definition}\lb{dW3.1}
Let $T$ be self-adjoint in $\cH$ and $m\in \bbN$ odd. Then $\Gamma_m(T)$ denotes the set of all 
self-adjoint operators $S$ in $\cH$ for which the containment
\begin{equation}\lb{W3.10}
\big[(S - z I_{\cH})^{-m} - (T - z I_{\cH})^{-m}\big] \in \cB_1(\cH), \quad z \in \bbC\backslash \bbR,
\end{equation}
holds. 
\end{definition}

One observes the following transitivity property:  if $B\in \Gamma_m(A)$ and $ C\in \Gamma_m(B)$, then $C \in \Gamma_m(A)$, as well.  In view of \eqref{W3.1}, one infers $B_0\in \Gamma_m(A_0)$ in the notation of Definition \ref{dW3.1}.

We note that for each $m \in \bbN$, $\Gamma_m(T)$ can be equipped with the family $\cD=\{d_{m,z}\}_{z\in\bbC\backslash \bbR}$ of pseudometrics (see \cite[Definition~IX.10.1]{Du66} for a precise definition) defined by  
\begin{equation}\lb{W3.11}
d_{m,z}(S_1,S_2) = \big\|(S_2-zI_{\cH})^{-m} - (S_1-zI_{\cH})^{-m}\big\|_{\cB_1(\cH)},\quad S_1,S_2\in \Gamma_m(T).
\end{equation}
For each fixed $\varepsilon >0$, $z\in \bbC\backslash \bbR$, and $S\in \Gamma_m(T)$, define 
\begin{equation}
B(S;d_{m,z},\varepsilon)=\{S'\in \Gamma_m(T)\, |\, d_{m,z}(S,S')<\varepsilon\},
\end{equation}
to be the $\varepsilon$-ball centered at $S$ with respect to the pseudometric $d_{m,z}$.  

\begin{definition}\lb{d4.4zaz}
$\cT_m(\cD,T)$ is the topology on $\Gamma_m(T)$ with the subbasis
\begin{equation}
\mathfrak{B}_m(\cD,T) = \{B(S;d_{m,z},\varepsilon)\,|\, S\in \Gamma_m(T), \, z\in \bbC\backslash\bbR,\, \varepsilon>0\}.
\end{equation}
That is, $\cT_m(\cD,T)$ is the smallest topology on $\Gamma_m(T)$ which contains $\mathfrak{B}_m(\cD,T)$.
\end{definition}

In order to state the main results of this section, we introduce one more hypothesis. 

\begin{hypothesis}\lb{h3.5}
$(i)$  Let $A_0$, $B_0$, and $B_1$ denote self-adjoint operators in $\cH$ with 
$B_0, B_1\in \Gamma_m(A_0)$ for some odd $m\in \bbN$, and let 
$\{B_{\tau}\}_{\tau\in [0,1]}\subset\Gamma_m(B_0)$ $($and hence in $\Gamma_m(A_0)$$)$ be a path from $B_0$ to $B_1$ in $\Gamma_m(B_0)$ depending continuously on $\tau\in[0,1]$ with respect to the topology  $\cT_m(\cD,T)$ introduced in Definition \ref{d4.4zaz}. \\
$(ii)$ Assume that $\varphi: \bbR\to \bbR$ satisfies the conditions in \eqref{W3.2}.
\end{hypothesis}

\begin{proposition}\label{prop_continuity}
Assume Hypothesis \ref{h3.5}. Then $\varphi(B_0)\in \Gamma_1(\varphi(A_0))$, and  
\begin{equation}
\{\varphi(B_{\tau})\}_{\tau\in[0,1]}\subset \Gamma_1(\varphi(B_0))
\end{equation}
is a path from $\varphi(B_0)$ to $\varphi(B_1)$ in $\Gamma_1(\varphi(B_0))$ depending 
continuously on $\tau\in[0,1]$ with respect to the metric $d_{1,i}(\, \cdot \,, \, \cdot \,)$.
\end{proposition}
\begin{proof}
The claims that $\varphi(B_0)\in \Gamma_1(\varphi(A_0))$ and $\{\varphi(B_{\tau})\}_{\tau\in[0,1]}\subset \Gamma_1(\varphi(B_0))$ follow immediately from \cite[Theorem~2.3]{Ya05}.  To establish continuity of the path $\{\varphi(B_{\tau})\}_{\tau\in[0,1]}$ with respect to the metric 
$d_{1,i}(\, \cdot \,, \, \cdot \,)$, an application of the estimate \eqref{the_estimate} yields the existence of a constant $C(\varphi) \in (0,\infty)$ and points $a_1,a_2 \in \bbR \backslash \{0\}$ such that
\begin{equation}\lb{W3.29}
\begin{split}
d_{1,i}(\varphi(B_{\tau}),\varphi(B_{\tau'}))\leq C(\varphi)\big(d_{m,a_1i}(B_{\tau},B_{\tau'}) 
+ d_{m,a_2i}(B_{\tau},B_{\tau'})\big),&\\
\tau,\tau'\in [0,1].&
\end{split}
\end{equation}
Thus, continuity of the path $\{\varphi(B_{\tau})\}_{\tau\in[0,1]}$ with respect to 
$d_{1,i}(\, \cdot \,, \, \cdot \,)$ follows by hypothesis from the continuity of $\{B_{\tau}\}_{\tau\in[0,1]}$ with respect to the topology $\cT_m(\cD,T)$.
\end{proof}

The following theorem represents the principal result of this section.

\begin{theorem}\lb{tW3.4}
Assume Hypothesis \ref{h3.5} and let $\xi_0(\,\cdot\,;\varphi(B_0),\varphi(A_0))$ be a spectral shift function for the pair $(\varphi(B_0),\varphi(A_0))$.~Then for each $\tau\in[0,1]$, there is a unique spectral shift function 
$\xi(\,\cdot\,;\varphi(B_{\tau}),\varphi(A_0))$ for the pair $(\varphi(B_{\tau}),\varphi(A_0))$ depending continuously on $\tau\in[0,1]$ in the $L^1(\bbR;(\lambda^2+1)^{-1}d\lambda)$-norm such that
\begin{equation}\lb{W3.26}
\xi(\,\cdot\,;\varphi(B_0),\varphi(A_0)) = \xi_0(\,\cdot\,;\varphi(B_0),\varphi(A_0)).
\end{equation}
Consequently,
\begin{equation}\lb{W3.27}
\xi(\,\cdot\,;B_{\tau},A_0):=\xi(\varphi(\cdot);\varphi(B_{\tau}),\varphi(A_0)),
\end{equation}
the corresponding spectral shift function for the pair $(B_{\tau},A_0)$, depends continuously 
on $\tau\in[0,1]$ in the $L^1(\bbR;(|\nu|^{m+1}+1)^{-1}d\nu)$-norm and satisfies
\begin{equation}\lb{W3.28}
\xi(\,\cdot\,;B_0,A_0) = \xi_0(\varphi(\cdot);\varphi(B_0),\varphi(A_0)).
\end{equation}
\end{theorem}
\begin{proof}
Let $\xi_0(\,\cdot\,;\varphi(B_0),\varphi(A_0))$ be a spectral shift function for the pair of operators $(\varphi(B_0),\varphi(A_0))$.  Since $\{\varphi(B_{\tau})\}_{\tau\in[0,1]}\subset \Gamma_1(\varphi(B_0))$ is a continuous path with respect to $d_{1,i}(\,\cdot\,,\,\cdot\,)$, an application of \cite[Lemma 8.7.5]{Ya92} yields for each pair $(\varphi(B_{\tau}),\varphi(A_0))$, $\tau\in [0,1]$, a unique spectral shift function $\xi(\,\cdot\,;\varphi(B_{\tau}),\varphi(A_0))\in L^1(\bbR;(\lambda^2+1)^{-1}d\lambda)$, depending continuously on $\tau\in[0,1]$ in the $L^1(\bbR;(\lambda^2+1)^{-1}d\lambda)$-norm and such that \eqref{W3.26} is satisfied.

For each $\tau\in[0,1]$, let $\xi(\,\cdot\,;B_{\tau},A_0)$ denote the spectral shift function for the pair 
$(B_{\tau},A_0)$ defined by \eqref{W3.27}.  Evidently, \eqref{W3.28} holds, and it only remains to establish continuity of $\xi(\,\cdot\,;B_{\tau},A_0)$ with respect to the 
$L^1(\bbR;(|\nu|^{m+1}+1)^{-1}d\nu)$-norm.  To this end, one applies \eqref{W3.27} and makes the change of variable in \eqref{W3.7}.  Consequently,
\begin{align}
&\int_{\bbR}\big|\xi(\nu;B_{\tau},A_0) - \xi(\nu;B_{\tau'},A_0)\big|\big(|\nu|^{m+1}+1\big)^{-1}\, d\nu\lb{W3.30}\\
&\quad = \int_{\bbR}\frac{\big|\xi(\mu;\varphi(B_{\tau}),\varphi(A_0)) - \xi(\mu;\varphi(B_{\tau'}),\varphi(A_0))\big|}{(|\varphi^{-1}(\mu)|^{m+1}+1)\varphi'(\varphi^{-1}(\mu))}\, d\mu.\no
\end{align}
Next, one obtains the following estimates on the weight of the measure on the right-hand side of the equality in \eqref{W3.30}:
\begin{equation}\lb{W3.31}
\frac{1}{(|\varphi^{-1}(\mu)|^{m+1}+1)\varphi'(\varphi^{-1}(\mu))}\leq \frac{C_0}{\mu^2+1},\quad \mu \in [-1,1],
\end{equation}
for some constant $C_0>0$, having used the last inequality in \eqref{W3.2}, and
\begin{equation}\lb{W3.32}
\begin{split}
&\frac{1}{(|\varphi^{-1}(\mu)|^{m+1}+1)\varphi'(\varphi^{-1}(\mu))}\\
&\quad = \frac{1}{m|\mu|^{1-1/m}\big(|\mu|^{1+1/m} +1\big)}\leq \frac{1}{\mu^2 + 1},\quad |\mu|>1.
\end{split}
\end{equation}
Combining \eqref{W3.30}, \eqref{W3.31}, and \eqref{W3.32}, and setting $C:=\max\{1,C_0\}$,
\begin{align}
&\int_{\bbR}\frac{\big|\xi(\nu;B_{\tau},A_0) - \xi(\nu;B_{\tau'},A_0)\big|}{|\nu|^{m+1}+1}\, d\nu\no\\
&\quad \leq C \int_{\bbR}\frac{\big|\xi(\mu;\varphi(B_{\tau}),\varphi(A_0)) - \xi(\mu;\varphi(B_{\tau'}),\varphi(A_0))\big|}{\mu^2+1}\, d\mu,\quad \tau,\tau'\in [0,1],\lb{W3.33}
\end{align}
and continuity of $\xi(\,\cdot\,;B_{\tau},A_0)$ in $L^1(\bbR;(|\nu|^{m+1}+1)^{-1}d\nu)$ follows from continuity of $\xi(\,\cdot\,;\varphi(B_{\tau}),\varphi(A_0))$ in $L^1(\bbR;(\mu^2+1)^{-1}d\mu)$.
\end{proof}

\begin{remark} \lb{r3.3}
If $\{\tau_n\}_{n=1}^{\infty}\subset[0,1]$ and $\tau_n\to 0$ as $n\to \infty$, then Theorem \ref{tW3.4} implies
\begin{equation}\lb{W3.34}
\lim_{n \to \infty} \|\xi(\, \cdot \, ; B_{\tau_n}, A_0) 
- \xi(\, \cdot \, ; B_0, A_0)\|_{L^1(\bbR; (|\nu|^{m+1} + 1)^{-1}d\nu)}=0.
\end{equation}
In particular, there exists a subsequence of $\{\xi(\, \cdot \, ; B_{\tau_n}, A_0)\}_{n \in \bbN}$ which converges pointwise a.e.\ to $\xi(\, \cdot \, ;  B_0, A_0)$ as $n \to \infty$. \hfill $\diamond$
\end{remark}

We conclude with an elementary consequence of Theorem \ref{tW3.4}.

\begin{corollary} \lb{c3.4} 
Assume Hypothesis \ref{h3.5}.~If $f \in L^{\infty}(\bbR)$, then
\begin{equation}
\lim_{\tau\to 0^+} \|\xi(\, \cdot \, ; B_{\tau}, A_0) f 
- \xi(\, \cdot \, ; B_0, A_0) f\|_{L^1(\bbR; (|\nu|^{m+1} + 1)^{-1}d\nu)} = 0,  \lb{W3.36}
\end{equation}
in particular,
\begin{equation}
\lim_{\tau \to 0^+} \int_{\bbR} \xi(\nu; B_{\tau}, A_0) d \nu \, g(\nu) 
= \int_{\bbR} \xi(\nu; B_0, A_0) d \nu \, g(\nu)     \lb{W3.37}
\end{equation}
for all $g \in L^{\infty}(\bbR)$ such that 
$\esssup_{\nu \in \bbR} \big|(|\nu|^{m+1} + 1) g(\nu)\big| < \infty$.
\end{corollary}

In the special case of one-dimensional systems, particularly, Schr\"odinger and Dirac-type operators 
on $\bbR$ or $(0,\infty)$ with sufficiently short-range potentials, the scattering phase shift is known  
to coincide with the spectral shift function (up to a constant factor) and continuity of scattering phase shifts with respect to the potential coefficient is known (see, \cite[Theorem~5.5]{BG83}). 

In conclusion, we note once more that in the special case $m=1$, the continuity result for spectral 
shift functions with respect to trace norm convergence of resolvent differences was derived by 
Yafaev \cite[Lemma~8.7.5]{Ya92}. The principal purpose of this section was to extend this result 
to higher odd integer powers $m$ of resolvents in order to make this continuity result available to 
$n$-dimensional elliptic partial differential operators (e.g., Schr\"odinger and Dirac-type operators) 
for which $m$ has to be chosen sufficiently large, depending on $n \in \bbN$.

\appendix
\section{The Case where $A$ and $B$ are Bounded From Below} \lb{sA}
\renewcommand{\theequation}{A.\arabic{equation}}
\renewcommand{\thetheorem}{A.\arabic{theorem}}
\setcounter{theorem}{0} \setcounter{equation}{0}

Due to its particular importance in applications (e.g., in connection with multi-dimensional 
Schr\"odinger operators), we now also briefly treat the case where $A$ and 
$B$ be are self-adjoint and bounded from below. In fact, without loss of generality, we assume throughout this appendix that $A$ and $B$ are strictly positive, self-adjoint operators in $\cH$, 
that is, for some $\varepsilon > 0$,
\begin{equation}
A \geq \varepsilon I_{\cH}, \quad B \geq \varepsilon I_{\cH}.    \lb{A.1}
\end{equation}

Since this case is significantly simpler than the case treated in Sections \ref{s2} and \ref{s3}, 
we primarily mention results without detailed proofs.
 
The symbol $\cJ^{A,B}_\phi$ is now of the form 
\begin{equation}\lb{A.2}
\cJ^{A,B}_\phi(T)=\int_{\bbR_+}\int_{\bbR_+}\phi(\lambda,\mu) \, 
d E_A(\lambda)\, T\, d E_B(\mu), \quad T\in \cB(\cH).
\end{equation} 
In addition, $\gM_p$, $1\leq p\leq \infty$, and $\|\phi\|_{\gM_p}=\|\cJ^{A,B}_\phi\|_{\cB_p(\cH)\to\cB_p(\cH)}$ are defined as in Section \ref{s2}, and again, we denote $\gM:=\gM_1=\gM_\infty$. As before, $\gM_p\subset \gM$ and $\|\phi\|_{\gM_p}\leq \|\phi\|_{\gM}$, $p \in [1,\infty)$.

Throughout this section we assume that $A$ and $B$ satisfy \eqref{A.1} and that 
\begin{equation}\lb{A.44}
T:= \big[( A + I_{\cH})^{-m} - ( B + I_{\cH})^{-m}\big] \in \cB_p(\cH), 
\; p \in [1,\infty) \text{ (resp., $\cB(\cH)$)}.
\end{equation}

The principal reason why the case of semibounded operators is significantly easier than the case treated in Sections \ref{s2} and \ref{s3} is the fact than both $( A + I_{\cH})^{-m}$ and $( B + I_{\cH})^{-m}$ are self-adjoint operators and therefore, one can use the fundamental results obtained in \cite{Pe85} and \cite{PS11}.

Let 
\begin{equation} 
\psi(\lambda):=\frac1{(\lambda+1)^m}, \quad \lambda\geq 0, 
\end{equation} 
assume $f$ is a bounded function on $\bbR_+$, and introduce $g:=f\circ \psi^{-1}$. 
One can write 
\begin{equation} 
f(A)-f(B)=(f\circ \psi^{-1})(\psi(A))-(f\circ \psi^{-1})(\psi(B))=\mathcal{J}^{A,B}_{g^{[1]}}(\psi(A)-\psi(B)). 
\end{equation} 
Since by hypothesis, $[\psi(A)-\psi(B)] \in \cB_p(\cH)$, $p \in [1,\infty)$ (resp., 
$[\psi(A)-\psi(B)] \in \cB(\cH)$), to establish the inclusion $[f(A)-f(B)] \in \cB_p(\cH)$, $p \in [1,\infty)$ 
(resp., $[f(A)-f(B)] \in \cB(\cH)$), it suffices to specify the class of functions rendering  
$\mathcal{J}^{A,B}_{g^{[1]}}$ a bounded operator on $\cB_p(\cH)$, $p \in [1,\infty)$ (resp., $\cB(\cH)$).

Let $p \in (1,\infty)$ and let $g$ be a (globally) Lipschitz function on $(0,1]$. Then \cite[Theorem~1]{PS11} guarantees that $g^{[1]}\in\gM_p$, and therefore, $[f(A)-f(B)] \in \cB_p(\cH)$, and 
\begin{equation} 
\|f(A)-f(B)\|_{\cB_p(\cH)}\leq C \big\|\big[( A + I_{\cH})^{-m} - ( B + I_{\cH})^{-m}\big]\big\|_{\cB_p(\cH)}, 
\end{equation} 
for some constant $C = C(p,f) \in (0,\infty)$, $p \in (1,\infty)$.  

Next, consider the case where \eqref{A.44} holds for $\cB_1(\cH)$ or $\cB(\cH)$. 
Let $g\in B_{\infty,1}^1$, 
where $B_{\infty,1}^1$ stands for a certain Besov class (see \cite[Sect.~6]{Pe85} for the precise definition). Then \cite[Theorem~8]{Pe85} implies that $g^{[1]}\in\gM$, that is, $[f(A)-f(B)] \in \cB_1(\cH)$ (resp.,$[f(A)-f(B)] \in \cB(\cH)$), and 
\begin{align}
\begin{split}  
& \|f(A)-f(B)\|_{\cB_1(\cH)}\leq C \big\|\big[( A + I_{\cH})^{-m} - ( B + I_{\cH})^{-m}\big]\big\|_{\cB_1(\cH)}, 
\\
& \big(\text{resp., } |f(A)-f(B)\|_{\cB(\cH)}\leq C \big\|\big[( A + I_{\cH})^{-m} - ( B + I_{\cH})^{-m}\big]\big\|_{\cB(\cH)}\big), 
\end{split}
\end{align} 
where $C = C(f) \in (0,\infty)$.

Thus, one arrives at the following result.
\begin{proposition}
Assume that \eqref{A.44} holds for $\cB_p(\cH)$, $p \in(1,\infty)$, and $\cB_1(\cH)$, or $\cB(\cH)$, respectively. In addition, let  $f\circ \psi^{-1}$ be a $($globally$)$ Lipschitz function on $(0,1]$ and 
$f\circ \psi^{-1}\in B_{\infty,1}^1$. Then,  
\begin{equation} 
[f(A)-f(B)] \in \cB_p(\cH), \quad p \in [1,\infty) \,\, (\text{resp., } 
[f(A)-f(B)] \in \cB(\cH)),   
\end{equation} 
and for some $C_p = C_p(f) \in (0,\infty)$, $p \in [1,\infty)$ $($resp., $C = C(f) \in (0,\infty)$$)$, 
\begin{align}\label{A.estimate} 
& \|f(A)-f(B)\|_{\cB_p(\cH)}\leq C_p \big\|\big[( A + I_{\cH})^{-m} - ( B + I_{\cH})^{-m}\big]\big\|_{\cB_p(\cH)}, \quad p \in [1,\infty),     \no \\
&\big(\text{resp., } \|f(A)-f(B)\|_{\cB(\cH)}\leq C \big\|\big[( A + I_{\cH})^{-m} - ( B + I_{\cH})^{-m}\big]\big\|_{\cB(\cH)}\big). 
\end{align} 
\end{proposition}

\begin{remark} \lb{rA.2}
$(i)$ Assume that \eqref{A.44} holds for $p \in (1,\infty)$ and let $f\circ \psi^{-1}\in B_{\infty,1}^1$. Since in this case $g^{[1]}\in\gM$, one also concludes that  $g^{[1]}\in\gM_p$ and $\|g^{[1]}\|_{\gM_p}\leq \|g^{[1]}\|_{\gM}$, $p \in (0,\infty)$. Therefore, the dependence of the constant in \eqref{A.estimate} on 
$p$ can be eliminated. \\ 
$(ii)$ On the other hand, if one is interested in the $p$-dependence of such a constant, the following can be asserted: Assume that \eqref{A.44} holds for $p \in (1,\infty)$ and let $f\circ \psi^{-1}$ be a (globally) Lipschitz function. It follows from \cite[Corollary~5.5]{CM-SPS14} that there exists a constant 
$C = C(f) \in (0, \infty)$, independent of $p$, such that 
\begin{equation} 
\|f(A)-f(B)\|_{\cB_p(\cH)}\leq C\frac{p^2}{p-1}
\big\|\big[( A + I_{\cH})^{-m} - ( B + I_{\cH})^{-m}\big]\big\|_{\cB_p(\cH)}, \quad p \in (1,\infty).
\end{equation} 
${}$ \hfill $\diamond$
\end{remark}

To have a result similar to Theorem \ref{tZ.10} we need to impose additional assumptions on the function $f$. 

\begin{theorem}\lb{tA.7}
Let $A,B$ be strictly positive self-adjoint operators on a Hilbert space $\cH$, and let the families 
$\bigl\{A_n \bigr\}_{n = 1}^\infty$ and $\bigl\{B_n \bigr\}_{n = 1}^\infty $ of strictly positive self-adjoint operators converging to $A$ and $B$, respectively, in the strong resolvent sense $($i.e., we 
assume \eqref{2.36A}$)$. Suppose that for fixed $m\in \bbN$, $m$ odd, and $p \in [1,\infty)$, 
\begin{equation}\lb{A.46}
\begin{split}
T &:= \big[( A + I_{\cH})^{-m} - ( B + I_{\cH})^{-m}\big] \in \cB_p(\cH),\\
T_n &:= \big[( A_n + I_{\cH})^{-m} - ( B_n + I_{\cH})^{-m}\big] \in
\cB_p(\cH), \quad n\in \bbN,
\end{split}
\end{equation}
and
\begin{equation} 
\lim_{n \rightarrow \infty} \|T_n - T\|_{\cB_p(\cH)}=0.   \lb{A.47} 
\end{equation}
Assume that $g=f\circ \psi^{-1}$ is a $($globally$)$ Lipschitz function on $(0,1]$ if $p>1$, and let 
$f\circ \psi^{-1}\in B_{\infty,1}^1$ if $p=1$. Assume, in addition, that $g'$ is a bounded function on $(0,1]$ satisfying a H\"older condition for some $\varepsilon>0$. Then 
\begin{align} 
& \lim_{n \rightarrow \infty}  \big\|\big[f(A_n) - f(B_n) \big] - \big[f(A)
- f(B)\big]\big\|_{\cB_p(\cH)}=0   \lb{A.48}.   
\end{align}
\end{theorem}
\begin{proof}
Writing 
\begin{align} 
\begin{split} 
& \big[f(A_n) - f(B_n) \big] - \big[f(A)- f(B)] 
= \mathcal{J}_{g^{[1]}}^{A_n,B_n}(T_n)-\mathcal{J}_{g^{[1]}}^{A,B}(T)     \\
& \quad = \mathcal{J}_{g^{[1]}}^{A_n,B_n}(T_n-T) 
+ \big(\mathcal{J}_{g^{[1]}}^{A,B}(T)-\mathcal{J}_{g^{[1]}}^{A,B}(T)\big). 
\end{split} 
\end{align}
The convergence of the first term on the right hand-side above can be proved as in Theorem \ref{tZ.10}. To prove the convergence of the second term, it is sufficient to show that $g^{[1]}\in\gA_l^s(E_B)\cap\gA_r^s(E_A)$. Since, by the assumption $g'$ is a bounded function  $g'$ satisfying the H\"older condition for some $\varepsilon>0$ \cite[Proposition 7.8 and Theorem 5.7]{BS73} imply that $g^{[1]}\in\gA_l^s(E_B)\cap\gA_r^s(E_A)$, and hence by Proposition \ref{pZ.7}  we conclude that 
\begin{equation} 
\lim_{n\to\infty}\big\|(\mathcal{J}_{g^{[1]}}^{A,B}(T)-\mathcal{J}_{g^{[1]}}^{A,B}(T)\big\|_{\cB_p(\cH)}=0.
\end{equation} 
\end{proof}


\noindent
{\bf Acknowledgments.} 
We are indebted to Dmitriy Zanin for his interest and insightful comments and would like 
to thank Marcus Waurick for discussions and helpful correspondence. We are particularly grateful 
to the anonymous referee for pointing out an error in our original proof of Theorem 2.12 which led to considerable modifications throughout this paper.   

 
\end{document}